\documentclass[a4 paper,11pt]{amsart}

\usepackage{amsmath} %
\usepackage{amssymb} %
\usepackage{amscd} %
\usepackage{amsthm} %
\usepackage{mathrsfs} %
\usepackage{euscript}
\usepackage{eufrak}
\usepackage{ulem}
\usepackage[alphabetic]{amsrefs}
\usepackage{ascmac} % 枠付き
\usepackage{comment}
\usepackage{bm}%太字
\usepackage[dvipdfmx]{graphicx} 
\usepackage{xcolor}
\usepackage{here}
\definecolor{unbleu}{rgb}{0.03, 0.15, 0.4}

\usepackage{hyperref}
\hypersetup{
pdfborder = {0 0 0},
colorlinks,
linkcolor=unbleu,
citecolor=unbleu,
urlcolor=unbleu
}

\usepackage{tikz}

\usepackage{fancyhdr}
 \newtheorem{theorem}{Theorem}[section]
 \newtheorem{lemma}[theorem]{Lemma}
 \newtheorem{proposition}[theorem]{Proposition}

 \newtheorem{question}[theorem]{Question}

\newtheorem{corollary}[theorem]{Corollary}

\theoremstyle{definition}
\newtheorem{definition}[theorem]{Definition}
\newtheorem{remark}[theorem]{Remark}
\newtheorem{example}[theorem]{Example}

\numberwithin{equation}{section}

\newcommand{\C}{\mathbb C}%     imaginary number
\newcommand{\R}{\mathbb R}%     real number
\title[Braids and the planar Newtonian N-body problem]
{A study of braids arising from simple choreographies of the planar Newtonian N-body problem}

\author[Y. Kajihara]{%
Yuika Kajihara}
\address{%
Department of Mathematics, Kyoto University, Kitashirakawa Oiwake-cho, Sakyo-ku,
Kyoto, 606-8502, Japan
}
\email{%
kajihara.yuika.6f@kyoto-u.ac.jp
} 

\author[E. Kin]{%
    Eiko Kin
}
\address{%
      Center for Education in Liberal Arts and Sciences, The University of Osaka, Toyonaka, Osaka 560-0043, Japan
}
\email{%
        kin.eiko.celas@osaka-u.ac.jp
}

\author[M. Shibayama]{%
    Mitsuru Shibayama
}
\address{%
Department of Applied Mathematics and Physics, 
Graduate School of Informatics, Kyoto University, 
Yoshida-Honmachi, Sakyo-ku, Kyoto 606-8501, Japan
}
\email{%
        shibayama@amp.i.kyoto-u.ac.jp}



\date{
April 25, 2025
}

\thanks{%
Y. K. was supported by JSPS KAKENHI Grant Number JP23H01081 and JP23K19009.
\\
E. K. was supported by JSPS KAKENHI Grant Number 
JP21K03247, JP22H01125, JP23H01081. 
\\
M. S. was supported by JSPS KAKENHI Grant Number 
JP23H01081, 
JP23K25778
}

\begin{document}
\dedicatory{Dedicated to the memory of Prof. Masaya Yamaguti on the 100th anniversary of his birth}

\begin{abstract}
We study periodic solutions of the planar Newtonian $N$-body problem 
with equal masses.
Each periodic solution traces out a braid with $N$ strands 
in 3-dimensional space.
When the braid is of pseudo-Anosov type, it has an associated stretch factor greater than 1, which reflects the complexity of the corresponding periodic solution. 
For each $N \ge 3$, Guowei Yu established the existence of a family of simple choreographies to the planar Newtonian $N$-body problem.
We prove that  braids arising from Yu’s periodic solutions 
are  of pseudo-Anosov types,
except in the special case where all particles move along a circle.
We also identify the simple choreographies whose braid types have the largest and smallest stretch factors, respectively. 
\end{abstract}
\maketitle

\section{Introduction}
\label{section_introduction}

We consider a  periodic solution 
$$\bm{z}(t)= (z_0(t), \dots, z_{N-1}(t)),\ z_i(t) \in {\Bbb R}^2\ 
(i=0, \dots, N-1)$$
of the planar Newtonian $N$-body problem with equal masses. 
Let $T>0$ be the period of the solution $\bm{z}(t)$. 
We take time to be a third axis orthogonal to the plane. 
For each fixed  $t_0 \in {\Bbb R}$, 
the trajectory of $\bm{z}(t)$ from $t_0$ to $t_0+ T$ 
traces a pure braid 
$b(\bm{z}([t_0,t_0+T]))$ (see Section~\ref{subsection_particle-dances}).
The following question  is a starting point for our study.

\begin{figure}[htbp]
\begin{center}
\includegraphics[width=3.8in]{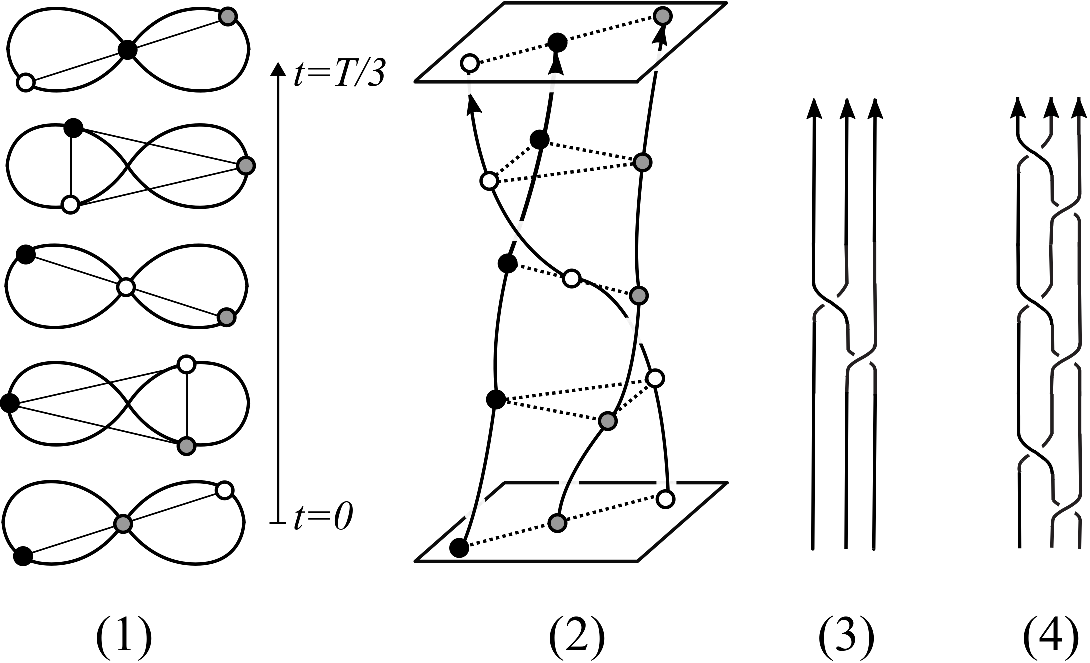}
\caption{(1) The figure-eight solution $\bm{z}(t)$ with period $T$. 
(2)(3) The primitive braid $b:= b(\bm{z}([0, \frac{T}{3}])) 
= \sigma_1^{-1} \sigma_2$. 
(4) The braid $b^3= (\sigma_1^{-1} \sigma_2)^3$ 
represents the braid type of the figure-eight.}
\label{fig_figure8choreography}
\end{center}
\end{figure}

\begin{question}[Montgomery \cite{Montgomery25} (cf. Moore \cite{Moore93})]
\label{question_Montgomery}
Is every pure braid type with $N$ strands realized by a 
periodic solution of the planar Newtonian $N$-body problem? 
\end{question}
See Definition~\ref{def_braid-type} for the notion of braid types.
Question~\ref{question_Montgomery} has been resolved for $N=3$, 
as shown by Moeckel-Montgomery \cite{MoeckelMontgomery15}, 
yet it remains wide open for $N \ge 4$.

A {\it simple choreography} of the planar $N$-body problem is a 
 periodic solution
 in which all $N$ particles chase each other along a single closed curve. 
We require that the phase shift between consecutive particles  be constant.
 If $\bm{z}(t)$ is a simple choreography with period $T$,  
then there exists a cyclic permutation 
$\sigma $ of $N$ elements $\{0, \dots, N-1\}$ 
such that 
$$z_i (t+ \tfrac{T}{N})= z_{\sigma(i)}(t) \hspace{2mm} 
\mbox{for}\ i \in \{ 0, \dots, N-1\} \ \mbox{and}\ t \in {\Bbb R}.$$
We call $\frac{T}{N}$ the {\it primitive period} of the 
simple choreography $\bm{z}(t)$. 
For any fixed $t_0 \in {\Bbb R}$, 
the trajectory of $\bm{z}(t)$ 
from $t_0$ to 
$t_0+ \frac{T}{N}$ determines a braid 
$b:= b(\bm{z}\big([t_0, t_0+ \frac{T}{N}]\big))$ 
that is called a {\it primitive braid} of $\bm{z}(t)$. 
We call its braid type 
the {\it primitive braid type} of $\bm{z}(t)$ (see Figure~\ref{fig_figure8choreography}).

Applying the Nielsen-Thurston classification of surface automorphisms 
\cite{Thurston88}, 
we classify braids  into three types: periodic, reducible, and pseudo-Anosov. For a braid of pseudo-Anosov type, there is an associated stretch factor greater than 1, which is a conjugacy invariant of the braid (see Section~\ref{subsection_Nielsen-Thurston}).  
We use stretch factors as a measure of the complexity of periodic solutions to the planar $N$-body problem.

For each integer $N \ge 3$, let us set 
$$\Omega_N= 
\{\bm{\omega}= (\omega_1, \dots, \omega_{N-1})\ | \ \omega_i \in \{1, -1\} \ \mbox{for}\  i \in \{1, \dots, N-1\}\}.$$ 
The main theorem of this paper is as follows. 

\begin{theorem}
\label{thm_main}
For each $N \ge 3$ and 
$\bm{\omega}= (\omega_1, \dots, \omega_{N-1}) \in \Omega_N$, 
there exists a simple choreography of the planar 
 Newtonian $N$-body problem 
 whose primitive braid type  is given by the braid 
 $\sigma_1^{\omega_1} \sigma_2^{\omega_2} \cdots \sigma_{N-1}^{\omega_{N-1}}$. 
In particular, the braid type of the simple choreography 
is given by 
$(\sigma_1^{\omega_1} \sigma_2^{\omega_2} \cdots \sigma_{N-1}^{\omega_{N-1}})^N$, 
and it is a pseudo-Anosov type if 
$\bm{\omega}$ is neither 
$(1, 1, \dots, 1)$ nor $(-1,-1, \dots, -1)$. 
Otherwise, the braid type of the simple choreography is periodic. 
\end{theorem}

\begin{figure}[htbp]
\begin{center}
\includegraphics[width=3in]{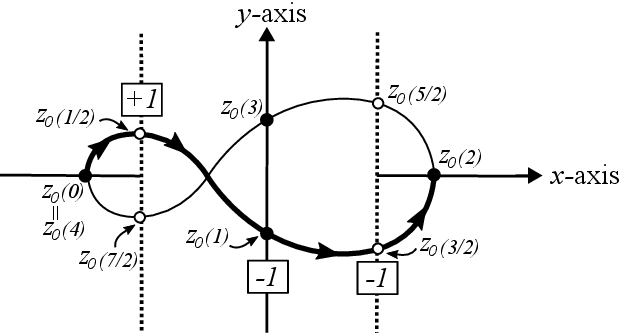}
\caption{
The closed curve obtained from the simple choreography 
$\bm{z}_{\bm{\omega}}(t)$ in the case $N=4$ and $\bm{\omega}= (1, -1,-1)$. 
Thick arrows illustrate 
the trajectory  of the $0$th particle $z_0(t)$  from 
$t=0$ to $\frac{N}{2}=2$. 
}
\label{fig_orbit_minimal4-braid}
\end{center}
\end{figure}

See Figure~\ref{fig_generator-sphere}(1) in 
Section~\ref{subsection_braid-mappingclass} 
for the generator $\sigma_i$ of the braid group. 
In 2017, Guowei Yu established the existence of a simple choreography 
$\bm{z}_{\bm{\omega}}(t) = (z_i(t))_{i=0}^{N-1}$ 
for each $\bm{\omega} \in \Omega_N$ 
to the planar Newtonian $N$-body problem with equal masses \cite{Yu17}. 
The period of the periodic solution $\bm{z}_{\bm{\omega}}(t)$ is $N$, 
and $\bm{z}_{\bm{\omega}}(t)$ fulfills 
$$z_{i}(t) = z_0(t+i) \ \ \text{for} \ t \in \mathbb{R} \ 
\text{and} \ i \in \{0,1,\dots,N-1\}.$$
Moreover, the closed curve on which $N$ particles travel 
is symmetric with respect to the $x$-axis. 
The element  $\bm{\omega}= (\omega_i)_{i=1}^{N-1} \in \Omega_N$ determines 
the shape of the closed curve (Figure~\ref{fig_orbit_minimal4-braid}). 
See Section~\ref{section_Yu} for more details. 
We will prove that 
the periodic solution $\bm{z}_{\bm{\omega}}(t)$ by Yu satisfies  
the statement of Theorem~\ref{thm_main}.

The figure-eight solution of the  planar  $3$-body problem \cite{ChenMont00} 
has the same braid type as the simple choreography $\bm{z}_{\bm{\omega}}(t)$ 
in the case $\bm{\omega} = (1, -1) $ (Example~\ref{ex_figure-8}). 
The super-eight solution of the planar $4$-body problem 
\cite{KZ03, Shibayama14} 
and the simple choreography $\bm{z}_{\bm{\omega}}(t)$ 
for $\bm{\omega} = (1, -1, 1)$ have the same braid type 
(Corollary~\ref{cor_super8}). 
See Figure~\ref{fig_Orb_n_4_seed_6}. 

\begin{figure}[htbp]
\begin{center}
\includegraphics[width=2.3in]{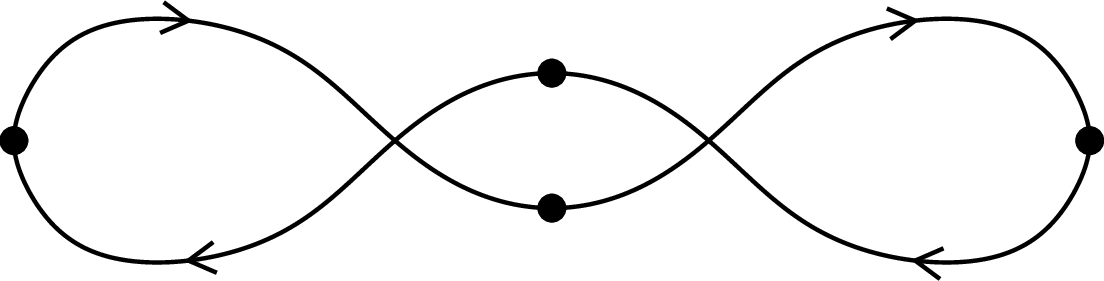}
\caption{The super-eight of the planar $4$-body problem.}
\label{fig_Orb_n_4_seed_6}
\end{center}
\end{figure}

To state the next result, 
we define 
$\bm{\omega}_{\max}, \bm{\omega}_{\min} \in \Omega_N$ 
as follows. 
\begin{eqnarray*}
\bm{\omega}_{\max}&:=& ((-1)^{i-1})_{i=1}^{N-1} 
= (1, -1, \dots, (-1)^{N-2}). 
\\
\bm{\omega}_{\min}&:=& (\omega_i)_{i=1}^{N-1}, \ 
\\
\mbox{where} \hspace{3mm}
 \omega_i&=& 
\left\{
\begin{array}{ll}
1
& \mbox{for}\  i= 1, \dots, \lfloor \frac{N}{2} \rfloor
\\
 -1
& \mbox{for}\ i=  \lfloor \frac{N}{2} \rfloor +1, \dots, N-1. 
\end{array}
\right.
\end{eqnarray*}
The simple choreography 
$\bm{z}_{\bm{\omega}_{\max}}(t)$ travels a chain made of $N-1$ loops. 
On the other hand, 
$\bm{z}_{\bm{\omega}_{\min}}(t)$ moves on a figure-eight curve and 
approximately half of the $N$ particles stay on each loop at every time. 
See Figure~\ref{fig:numsoln=19}.

\begin{figure}[htbp]
\begin{center}
  \begin{minipage}[b]{0.48\columnwidth}
  
\vspace{1.5cm}

  \centering
\includegraphics[scale=0.35]{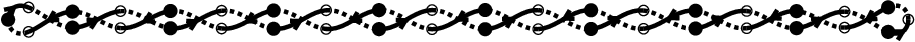}

(1) $\bm{z}_{\bm{\omega}}(t)$ for $\bm{\omega}= \bm{\omega}_{\max}$. 
\end{minipage}
 \begin{minipage}[b]{0.48\columnwidth}

 \vspace{1.5cm}
  \centering
\includegraphics[scale=0.35]{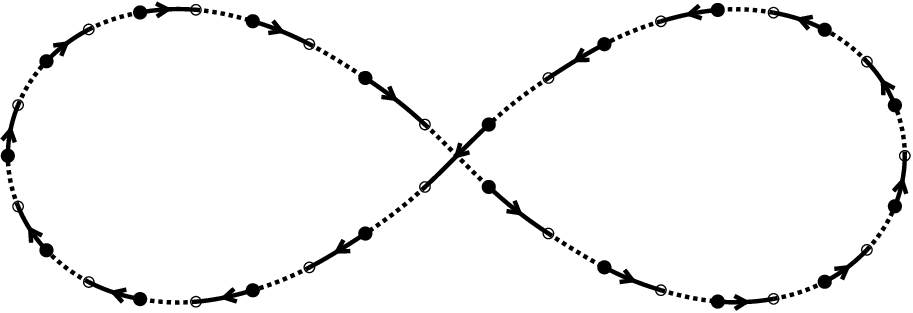} 

(2) $\bm{z}_{\bm{\omega}}(t)$ for $\bm{\omega}= \bm{\omega}_{\min}$. 
\end{minipage}
\caption{
Simple choreographies $\bm{z}_{\bm{\omega}}(t)$ in the case $N=19$. 
The dots denote the initial condition. 
The arrows indicate the trajectory of $\bm{z}_{\bm{\omega}}(t)$ 
from $t=0$ to $\frac{1}{2}$.}
\label{fig:numsoln=19}
\end{center}
\end{figure}

\begin{theorem}
\label{thm_omega_max-min}
Among all $\bm{\omega} \in \Omega_N$ except for
the two elements $(1,1, \dots,1)$ and $(-1,-1,\dots, -1)$, 
the simple choreography $\bm{z}_{\bm{\omega}}(t)$ 
whose braid type having the largest  stretch factor 
is realized  by $\bm{\omega}_{\max}$, 
while the one whose  braid type having the 
smallest  stretch factor is realized by 
$\bm{\omega}_{\min}$. 
\end{theorem}

A {\it multiple choreography} of the planar $N$-body problem is a 
 periodic solution such that 
 particles travel on $k$ different closed curves for some $k>1$. 
In \cite{Shibayama06}, 
the third author established  
the existence of a family of multiple choreographies 
to the planar Newtonian $2N$-body problem. 
These periodic solutions have pseudo-Anosov braid types 
whose stretch factors are quadratic irrational 
\cite{KajiharaKinShibayama23}. 
In this sense, the multiple choreographies given in \cite{Shibayama06}  
have an algebraic restriction from view points of pseudo-Anosov stretch factors. 
On the other hand, simple choreographies in Theorem~\ref{thm_main} 
do not have such a restriction. 
In particular, the simple choreography $\bm{z}_{\bm{\omega}}(t)$ 
in the case $\bm{\omega}= (-1,1,1)$ gives us the following result.

\begin{corollary}
\label{cor_minimal-4braid}
There exists a simple choreography 
of the planar Newtonian $4$-body problem 
 whose primitive braid type  is given by the pseudo-Anosov braid 
 $\sigma_1 \sigma_2^{-1} \sigma_3^{-1}$. 
 The primitive braid type of the simple choreography  
 has the stretch factor 
 $\approx 2.2966$ 
 which is the largest real root of the  polynomial 
 $t^4-2t^3-2t+1$ with degree $4$. 
\end{corollary}

For other periodic solutions that have been proven to exist, see, 
for example, \cite{Chen03a,Chen03b,Chen08,FerrarioTerracini04}. 
See also \cite{Simo01,Doedeletal03,Suvakov13} 
for periodic solutions that have been obtained numerically.

The organization of the paper is as follows. 
In Section~\ref{section_preliminaries}, 
we recall basic results on the braid groups and mapping class groups. 
In Section~\ref{section_Yu}, we review  the simple choreographies by Yu. 
Section~\ref{section_proofs} and Section~\ref{section_conclusion} 
contain the proofs of results and the conclusion in this paper. 

\subsection*{Acknowledgment}  
The second author would like to express her deep gratitude to Prof. Masaya Yamaguti, whose words have long guided her approach to mathematics:
``Mathematics is not meant only for geniuses." 
``What matters is not whether one has a talent for mathematics,
but rather to consider what we ought to do now,
and what is most important at this very moment.
As long as we keep that in mind, we can move forward without fear."

\section{Preliminaries} 
\label{section_preliminaries}

\subsection{Braid groups and mapping class groups}
\label{subsection_braid-mappingclass}
We review the basics of the braid groups. 
See also  \cite[Chapters~1, 4]{Birman74}. 
Let $B_n$ be the braid group of $n$ strands 
generated by $\sigma_1, \dots, \sigma_{n-1} $. 
The group $B_n$ has  the following presentation. 
$$
B_n=\left\langle \sigma_1, \dots, \sigma_{n-1}
\left|
\begin{matrix}
\ \sigma_i \sigma_j= \sigma_j \sigma_i \hspace{12mm} \mbox{if}\ |i-j|>1
 \\
\ \sigma_i \sigma_j \sigma_i = \sigma_j \sigma_i \sigma_j \hspace{5mm} 
 \mbox{if}\ |i-j|=1
\end{matrix}
\right.
\right\rangle
$$
The generator $\sigma_i$ corresponds to a geometric braid 
as in Figure~\ref{fig_generator-sphere}(1). 
There is a surjective homomorphism 
\begin{equation}
\label{equation_symmetrygroup}
  \hat{s}: B_n \to S_n  
\end{equation}
from  $B_n$ to the symmetry group $S_n$ 
of $n$ elements sending each $\sigma_j$ to the transposition $(j, j+1)$. 
The kernel of $\hat{s}$ is called the \textit{pure braid group} 
$P_n < B_n$. 
An element of $P_n$ is called a \textit{pure braid}.  

\begin{figure}[htbp]
\begin{center}
\includegraphics[width=3.8in]{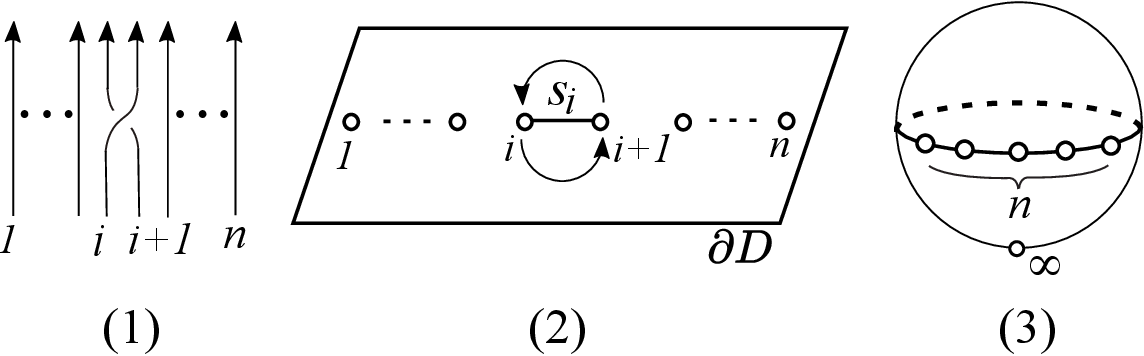}
\caption{
(1) $\sigma_i \in B_n$. 
(2) $h_i = \Gamma(\sigma_i) \in \mathrm{MCG}(D_n)$. 
(3) $\Sigma_{0, n+1}$. 
}
\label{fig_generator-sphere}
\end{center}
\end{figure}

Let $Z(B_n)$ be the center of the $n$-braid group $B_n$. 
The subgroup $Z(B_n)$ is an  infinite cyclic group 
generated by the full twist  
$\Delta^2 \in B_n$ (\cite[Corollary~1.8.4]{Birman74}), 
where $\Delta  \in B_n$ is the half twist 
that is defined by 
$$\Delta= (\sigma_1 \sigma_2 \dots \sigma_{n-1}) 
(\sigma_1 \sigma_2 \dots \sigma_{n-2}) \dots 
(\sigma_1 \sigma_2) \sigma_1.$$ 
Note that  $\Delta^2$ is obtained by rotating 
the set of $n$ points  one full revolution.

Let $\Sigma$ be an orientable, connected surface, 
possibly with punctures and boundary. 
The \textit{mapping class group} $\mathrm{MCG}(\Sigma)$ of $\Sigma$ 
is the group of 
isotopy classes of orientation preserving homeomorphisms of $\Sigma$ 
which preserve the punctures and boundary setwise. 
Let $\Sigma_{g,n}$ be an orientable, connected surface of genus $g$ and $n$ punctures. 
In this paper, we consider the mapping class groups of an 
$n$-punctured disk $D_n$ and an $n$-punctured sphere $\Sigma_{0,n}$. 
The group $\mathrm{MCG}(D_n)$ is generated by 
$h_1, \dots, h_{n-1}$, 
where $h_i$ is the right-handed half twist about  a segment $s_i$ 
connecting the $i$th puncture and $i+1$th puncture, i.e., 
$h_i$  interchanges the $i$th puncture and $i+1$th puncture  as in Figure~\ref{fig_generator-sphere}(2). 
A relation between $B_n$ and $\mathrm{MCG}(D_n)$ is given by 
the following surjective homomorphism 
\begin{equation}
\label{equation_Gamma}
\Gamma: B_n \rightarrow \mathrm{MCG}(D_n) 
\end{equation}
which sends $\sigma_i$ to $h_i$ for each $i \in \{1, \dots, n-1\}$. 
The kernel of $\Gamma$ is the center $Z(B_n)$ of $B_n$. 

\begin{definition}
\label{def_braid-type} 
The \textit{braid type} $\langle b \rangle$ of a braid $b \in B_n$ 
is the conjugacy class of $\Gamma(b)$ in $\mathrm{MCG}(D_n)$. 
Since $\mathrm{MCG}(D_n)$ is isomorphic to $B_n/Z(B_n)$, 
the braid type $\langle b \rangle$ can be identified with a conjugacy class 
in $B_n/Z(B_n)$. 
We may call the braid type of a pure braid 
the \textit{pure braid type}. 
\end{definition}

\subsection{Nielsen-Thurston classification}
\label{subsection_Nielsen-Thurston}

We assume that $3g-3+n \ge 1$. 
According to the Nielsen-Thurston classification of 
surface automorphisms \cite{Thurston88}, 
elements of $\mathrm{MCG}(\Sigma_{g,n})$ are classified into three types: periodic, reducible and pseudo-Anosov as we recall now. 
A mapping class $\phi \in \mathrm{MCG}(\Sigma_{g,n})$ is \textit{periodic}  
if $\phi$ is of finite order. 
A simple closed curve $C$ in $\Sigma_{g,n}$ is \textit{essential} 
if it is not homotopic to a point (possibly a point corresponding to a puncture).  
A mapping class 
$\phi \in \mathrm{MCG}(\Sigma_{g,n})$ is \textit{reducible} 
if there is a  collection of mutually disjoint and non-homotopic essential simple closed curves 
$C_1, \dots, C_j$ in $\Sigma_{g,n}$ (possibly $j=1$) such that 
$C_1 \cup \dots \cup C_j$ is preserved by $\phi$. 
Notice that there is a mapping class that is periodic and reducible.  
A mapping class $\phi \in \mathrm{MCG}(\Sigma_{g,n})$ is 
\textit{pseudo-Anosov} 
if $\phi$ is neither periodic nor reducible.  
The Nielsen-Thurston type is a conjugacy invariant, 
i.e., 
two mapping classes are conjugate to each other in $\mathrm{MCG}(\Sigma_{g,n})$, 
then their Nielsen-Thurston types are the same. 

We review properties of pseudo-Anosov mapping classes. 
For more details, see \cite{FLP,FarbMargalit12,Boyland94}. 
A homeomorphism $\Phi: \Sigma_{g,n} \rightarrow \Sigma_{g,n}$ is a 
\textit{pseudo-Anosov map} 
if there exist a constant $\lambda= \lambda(\Phi)>1$ and 
a pair of transverse measured foliations 
$(\mathcal{F}^+, \mu^+)$ and 
$(\mathcal{F}^-, \mu^-)$  so that 
$\Phi$ preserves both foliations $\mathcal{F}^+$ and $\mathcal{F}^-$, 
 and it contracts the leaves of $\mathcal{F}^-$ by $\frac{1}{\lambda}$ and 
 it expands the leaves of $\mathcal{F}^+$ by $\lambda$. 
More precisely, $\Phi$ fulfills 
$$\Phi((\mathcal{F}^+,  \mu^+)) = 
(\mathcal{F}^+, \lambda \mu^+)\ \hspace{2mm} \mbox{and}\  \hspace{2mm}
 \Phi((\mathcal{F}^-,  \mu^-)) = (\mathcal{F}^-, \tfrac{1}{\lambda} \mu^-).$$ 
The constant $\lambda>1$ is called the \textit{stretch factor} of $\Phi$. 
For each pseudo-Anosov mapping class $\phi \in \mathrm{MCG}(\Sigma_{g,n})$, 
there exists  a pseudo-Anosov homeomorphism 
$\Phi: \Sigma_{g,n} \rightarrow \Sigma_{g,n}$  
that is a representative of $\phi$. 
The \textit{stretch factor} $\lambda(\phi)$ of $\phi$ is defined by 
$\lambda(\phi)= \lambda(\Phi) $, 
and it is a conjugacy invariant of pseudo-Anosov mapping classes. 
We understand that stretch factors measure the complexity of pseudo-Anosov mapping classes.

A square matrix $M$ with nonnegative integer entries 
is \textit{Perron-Frobenius} 
if some power of $M$ is a positive matrix. 
In this case, 
the Perron-Frobenius theorem \cite[Theorem 1.1]{Seneta06} 
tells us that 
$M$ has a real eigenvalue 
$\lambda(M)>1$ which exceeds the moduli of all other eigenvalues. 
We call $\lambda(M)$ the \textit{Perron-Frobenius eigenvalue}. 
It is known that 
the stretch factor  $ \lambda(\phi)$ of a pseudo-Anosov mapping class 
$\phi$ is the largest eigenvalue of a Perron-Frobenius matrix, 
that is $ \lambda(\phi)$ is a Perron number. 
Observe that 
if $\phi$ is a pseudo-Anosov mapping class, then 
$\phi^k$ is pseudo-Anosov for all $k \ge 1$, and 
it holds  
\begin{equation}
\label{equation_stretchfactor-power}
\lambda(\phi^k)= (\lambda(\phi))^k.
\end{equation}

We recall the homomorphism 
$\Gamma: B_n \rightarrow \mathrm{MCG}(D_n)$ as in (\ref{equation_Gamma}). 
Collapsing the boundary of the disk to a point $\infty$ in the sphere, 
we obtain the $n+1$-punctured sphere $\Sigma_{0,n+1}$ 
(see Figure~\ref{fig_generator-sphere}(3)) and 
a homomorphism 
$$\mathfrak{c}: \mathrm{MCG}(D_n) \rightarrow \mathrm{MCG}(\Sigma_{0, n+1}).$$ 
We may identify a mapping class $\Gamma(b) \in \mathrm{MCD}(D_n) $ 
for  $b \in B_n$
with $\mathfrak{c}(\Gamma(b)) \in \mathrm{MCG}(\Sigma_{0, n+1})$. 
We say that a braid $b \in B_n$ is \textit{periodic} 
(resp. \textit{reducible}, \textit{pseudo-Anosov}) 
if the mapping class $\mathfrak{c}(\Gamma(b))$ is of the corresponding type. 
When  $b$ is  pseudo-Anosov, 
its \textit{stretch factor} $\lambda(b)$  is defined by the stretch factor 
of the pseudo-Anosov mapping class $\mathfrak{c}(\Gamma(b))$. 
In this case, it makes sense to say that the braid type $\langle b \rangle$ is 
pseudo-Anosov, since the Nielsen-Thurston type is a conjugacy invariant. 
The \textit{stretch factor}   
$\lambda(\langle b \rangle)$ of the braid type $\langle b \rangle$ can be defined  by $\lambda(\langle b \rangle)=\lambda(b)$.

\subsection{Braids as particle dances}
\label{subsection_particle-dances}
We consider the motion of $N$ points in the plane ${\Bbb R}^2$ 
$${\bm z}(t)= (z_0(t), \dots, z_{N-1}(t)),\ 
z_i(t) \in {\Bbb R}^2\ 
(i=0, \dots, N-1)$$ 
where $z_i(t) \in {\Bbb R}^2$ is the position of the $i$th point at 
 $t \in {\Bbb R}$. 
We assume the following conditions. 
\begin{itemize}
\item 
(collision-free) 
$z_i (t) \ne z_j(t)$ for $i \ne j$ and $t \in {\Bbb R}$. 

\item 
(periodicity) 
There exists  $T >0$ such that 
$$\{z_0(t), \dots, z_{N-1}(t)\}= \{z_0(t+T), \dots, z_{N-1}(t+T)\} \hspace{5mm} \mbox{for}\ t \in {\Bbb R}.$$
\end{itemize}
We take time to be a third axis orthogonal to the plane. 
Fixing $t_0 \in {\Bbb R}$, we have mutually disjoint $N$ curves 
\begin{eqnarray*}
[t_0,t_0+T] &\rightarrow& {\Bbb R}^2 \times {\Bbb R}
\\
t &\mapsto& (z_i(t), t). 
\end{eqnarray*}
The union of the curves forms a braid, denoted by $b(\bm{z}([t_0, t_0+T]))$. 
Such a braid is sometimes referred to as a \textit{particle dance}.

We turn to a  periodic solution 
$\bm{z}(t)= (z_0(t), \dots, z_{N-1}(t))$ 
of the planar Newtonian $N$-body problem. 
Suppose that 
the periodic solution $\bm{z}(t)$ has a  period $T$, 
that is $T$ is the smallest positive number such that 
$z_i(t)= z_i(t+T)$ for all $i= 0, \dots, N-1$ and $t \in {\Bbb R}$. 
Choosing any $t_0 \in {\Bbb R}$, 
we obtain a pure braid $b(\bm{z}([t_0,t_0+T]))$ with $N$ strands 
that obviously depends on the choice of $t_0$. 
Although the set of base points 
$\{z_0(t_0), \dots, z_{N-1}(t_0)\}$ of $b(\bm{z}([t_0,t_0+T]))$ 
does not necessarily lie along a straight line in the plane, it still makes sense to consider its braid type 
$\big \langle b(\bm{z}([t_0,t_0+T])) \big \rangle $. 
See \cite[Section~3.1]{KajiharaKinShibayama23}. 
Such a braid type $\big \langle b(\bm{z}([t_0,t_0+T])) \big \rangle $ 
does not depend on the choice of $t_0$ and 
we call it  
the \textit{braid type of the periodic solution} $\bm{z}(t)$. 

Suppose that $\bm{z}(t)$ is a simple choreography with period $T$ 
of the planar Newtonian $N$-body problem. 
Choosing any $t_0 \in {\Bbb R}$, 
we have 
$$\{z_0(t_0), \dots, z_{N-1}(t_0)\}= 
\{z_0(t_0+\tfrac{T}{N}), \dots, z_{N-1}(t_0+\tfrac{T}{N})\},$$ 
and we obtain a braid 
$b:= b(\bm{z}\big([t_0, t_0+ \frac{T}{N}]\big))$ 
as a particle dance. 
We call $\frac{T}{N}$ the \textit{primitive period} of the 
simple choreography $\bm{z}(t)$ and call 
the braid $b(\bm{z}\big([t_0, t_0+ \frac{T}{N}]\big))$ 
a \textit{primitive braid} of  $\bm{z}(t)$. 
Its braid type 
$\big \langle b(\bm{z}\big([t_0, t_0+ \frac{T}{N}]\big)) \big \rangle$
is called the \textit{primitive braid type} of the simple choreography 
$\bm{z}(t)$. 
Clearly, the $N$th powers $b^N$ of 
the primitive braid $b$ equals $b(\bm{z}([t_0,t_0+T]))$, 
and $b^N$ represents the  braid type 
$\big \langle b(\bm{z}([t_0,t_0+T])) \big \rangle $ of the periodic solution 
$\bm{z}(t)$.

\subsection{Compositions of integers} 
A \textit{composition} of a positive integer $n$ is a representation 
of $n$ as a sum of positive integers. 
For example, 
there are four compositions of $3$:
$$3= 3,\ 3= 1+2,\ 3= 2+1, 3= 1+1+1.$$
(There is another notion `partition' of a positive integer. 
A \textit{partition} of  $n$ is a representation of $n$ 
as a sum of positive integers, 
where the order of the summands is \textit{not} taken into account.)  

A composition of $n$ is written by a $(k+1)$-tuple 
$\bm{m}= (m_1, \cdots, m_{k+1})$ 
of positive integers $m_i$ with 
$n= \sum_{i=1}^{k+1}m_i$ and $k \ge 0$. 
Let 
$\Psi_n$ denote the set of all compositions of $n$. 
For example, $\Psi_3$ consists of the four compositions 
$$(3),\  (1,2),\  (2,1),\  (1,1,1).$$
The cardinality of the set 
$\Psi_n$ is $2^{n-1}$. 
To see this, 
consider $n$ circles. 
$$\circ_1 \hspace{5mm} \circ_2 \hspace{5mm} \circ_3 \hspace{5mm} \cdots 
\hspace{5mm}  \circ_n$$
There are $n-1$ spaces between them. 
In each space, one can choose to place a bar or leave it empty. 
There exist $2^{n-1}$ possible ways to place the bars, 
and each configuration corresponds to a composition of $n$. 
This means that 
the number of compositions of $n$ is $2^{n-1}$. 
For example, 
the four compositions of $3$ are written by 
\begin{equation}
\label{equation_four-compositions}
\circ\ \circ \ \circ \leftrightarrow (3),\hspace{2mm} 
\circ| \circ \ \circ \leftrightarrow (1,2) ,\hspace{2mm} 
\circ\ \circ|\circ \leftrightarrow (2,1),\hspace{2mm} 
\circ| \circ| \circ \leftrightarrow (1,1,1) 
\end{equation}

We turn to another set that is related to compositions of integers. 
Recall  the set $\Omega_N$ as in Section~\ref{section_introduction}. 
Given an integer $N \ge 3$, 
let $\Omega_N^+$ denote the set of all elements 
$\bm{\omega}= (\omega_1, \dots, \omega_{N-1}) \in \Omega_N$ 
whose first component $\omega_1$ is $1$:
$$\Omega_N^{+}:= \{\bm{\omega}= (\omega_1, \dots, \omega_{N-1}) 
\in \Omega_N\ | \ \omega_1= 1\}.$$

\begin{lemma}
\label{lem_bijection} 
There is a bijection $\Theta: \Psi_{N-1} \rightarrow \Omega_N^+$. 
\end{lemma}

\begin{proof}
We represent a composition $\bm{m} \in \Psi_{N-1}$ 
by using $N-1$ circles together with bars. 
We now define $\Theta({\bm{m}}) = (\omega_1=1, \omega_2 \dots, \omega_{N-1})$. 
For each $i \in \{1, \dots, N-1\}$,
we replace the $i$th circle with
$1$ or $-1$ which indicates
$\omega_i=1$ or $\omega_i=-1$ 
as we explain now. 
We set $\omega_1= 1$. 
Suppose that we have replaced the $i$th circle with $1$ or $-1$. 
(Then $\omega_i$ is determined.)  
If there is a bar between $i$th circle and $(i+1)$th circle, 
then we set $\omega_{i+1}= - \omega_{i}$. 
Otherwise, we set $\omega_{i+1}= \omega_i$. 
In other words, if 
$\bm{m}= (m_1, \dots, m_{k+1})$ is an element of $\Psi_{N-1}$, then  
$$\Theta(\bm{m}) = (\underbrace{1,\dots, 1}_{m_1},
\underbrace{-1,\dots, -1}_{m_2}, \dots, 
\underbrace{(-1)^{k},\dots, (-1)^{k}}_{m_{k+1}}) \in \Omega_{N}^+.$$
Clearly, $\Theta: \Psi_{N-1} \rightarrow \Omega_N^+$ is injective. 
To see $\Theta: \Psi_{N-1} \rightarrow \Omega_N^+$ is surjective, 
we take any 
$\bm{\omega}= (\omega_1 =1, \omega_2, \dots, \omega_{N-1}) \in \Omega_{N}^+$. 
By applying the reverse operation, 
one can determine a composition $\bm{m}_{\bm{\omega}} \in \Psi_{N-1}$ 
so that $\Theta(\bm{m}_{\bm{\omega}}) = \bm{\omega}$. 
This completes the proof. 
\end{proof}

For example,  
the bijection $\Theta: \Psi_3 \rightarrow \Omega_4^+$ is described as follows. 
\begin{eqnarray*}
\Theta: \hspace{3mm}\Psi_3 &\rightarrow& \Omega_4^+
\\
\circ\ \circ \ \circ &\mapsto& (1,1,1), 
\\
\circ| \circ \ \circ&\mapsto& (1, -1,-1), 
\\
\circ\ \circ|\circ&\mapsto&(1,1,-1), 
\\
\circ| \circ| \circ&\mapsto&(1,-1,1). 
\end{eqnarray*}

We write the $n$-tuple 
$ (1,1, \cdots, 1) \in \Psi_n$ correspoinding to the composition 
$n= 1+ 1+ \cdots +1$ as $ \bm{1}_n$. 

\begin{example}
\label{ex_correspondence}
Consider the bijection $\Theta: \Psi_{N-1} \rightarrow \Omega_N^+$. 
\begin{enumerate}
\item 
The image of $\bm{1}_{N-1}$ under $\Theta$ is 
$$\Theta(\bm{1}_{N-1}) =((-1)^{i-1})_{i=1}^{N-1}= (1, -1, \dots, (-1)^{N-2}).$$
\item 
The image of 
the composition $ (N-1)$ under $\Theta$ is 
$ (1,1, \dots, 1)$. 
\end{enumerate}
\end{example}

\subsection{Braids associated with compositions of integers}
\label{subsection_pAfamily}

We introduce braids 
$\alpha_{\bm{\omega}}$, $e_{\bm{\omega}}$, $o_{\bm{\omega}}$ 
for each $\bm{\omega} \in \Omega_N$ 
and $\beta_{\bm{m}}$ for each composition $\bm{m} \in \Psi_{N-1}$. 
We first define an $N$-braid $\alpha_{\bm{\omega}}$ 
for $\bm{\omega} = (\omega_1, \dots, \omega_{N-1})$
as follows. 
\begin{equation}
\label{equation_alpha}
\alpha_{\bm{\omega}}: = \sigma_1^{\omega_1} \sigma_2^{\omega_2} \cdots \sigma_{N-1}^{\omega_{N-1}}.    
\end{equation}
Notice that the $N$th power $\alpha_{\bm{\omega}}^N$ of $\alpha_{\bm{\omega}}$ is a pure braid. 
We next define $N$-braids 
$e_{\bm{\omega}}$ and $o_{\bm{\omega}}$ as follows: 
\begin{equation}
\label{equation_3kinds_of_braids}
e_{\bm{\omega}}:= \prod_{\substack{i \in \{1, \dots,N-1\} \\ i\ \text{even}  }} \sigma_i^{\omega_i}, \hspace{5mm}
o_{\bm{\omega}}:= \prod_{\substack{i \in \{1, \dots,N-1\} \\ i\ \text{odd}  }} \sigma_i^{\omega_i}. 
\end{equation}

\begin{example}
\label{ex_N7}
For each $\bm{\omega} \in \Omega_7$, 
we demonstrate that 
$ e_{\bm{\omega}} o_{\bm{\omega}}$ is conjugate to $\alpha_{\bm{\omega}}$ 
in $B_7$. 
Recall that the braid group has a relation 
$\sigma_i \sigma_j= \sigma_j \sigma_i$ if $|i-j|>1$. 
We set 
$b_0:= e_{\bm{\omega}}^{-1} (e_{\bm{\omega}} o_{\bm{\omega}})e_{\bm{\omega}} 
= o_{\bm{\omega}}  e_{\bm{\omega}}$. 
Then we have 
$$b_0= o_{\bm{\omega}} \cdot e_{\bm{\omega}}
=
\sigma_5^{\omega_5} \sigma_3^{\omega_3} \sigma_1^{\omega_1} 
\cdot 
\sigma_6^{\omega_6} \sigma_4^{\omega_4} \sigma_2^{\omega_2} 
=
(\sigma_5^{\omega_5} \sigma_6^{\omega_6})  
(\sigma_3^{\omega_3} \sigma_4^{\omega_4} 
\sigma_1^{\omega_1} \sigma_2^{\omega_2}). $$
Set 
$b_1:= (\sigma_5^{\omega_5} \sigma_6^{\omega_6})^{-1} b_0 
(\sigma_5^{\omega_5} \sigma_6^{\omega_6})$. 
We have 
$$b_1= (\sigma_3^{\omega_3} \sigma_4^{\omega_4} 
\sigma_1^{\omega_1} \sigma_2^{\omega_2}) 
(\sigma_5^{\omega_5} \sigma_6^{\omega_6} ) 
= 
(\sigma_3^{\omega_3} \sigma_4^{\omega_4} 
\sigma_5^{\omega_5} \sigma_6^{\omega_6})  
(\sigma_1^{\omega_1} \sigma_2^{\omega_2}).$$
Set $b_2:=(\sigma_3^{\omega_3} \sigma_4^{\omega_4} 
\sigma_5^{\omega_5} \sigma_6^{\omega_6})^{-1} b_1 
(\sigma_3^{\omega_3} \sigma_4^{\omega_4} 
\sigma_5^{\omega_5} \sigma_6^{\omega_6})$ 
that is of the form 
$$b_2
=\sigma_1^{\omega_1} \sigma_2^{\omega_2}\sigma_3^{\omega_3} \sigma_4^{\omega_4} 
\sigma_5^{\omega_5} \sigma_6^{\omega_6}
= \alpha_{\bm{\omega}}.$$
Thus, $e_{\bm{\omega}} o_{\bm{\omega}}$, $b_0$, $b_1$, 
and $b_2= \alpha_{\bm{\omega}}$ 
are conjugate to each other in $B_7$. 
\end{example}

The following lemma is used in the proof of Theorem~\ref{thm_braid-type-omega}. 

\begin{lemma} 
\label{lem_oe-alpha}
For each $\bm{\omega} \in \Omega_N$, 
the braids $e_{\bm{\omega}} o_{\bm{\omega}}$ and $\alpha_{\bm{\omega}}$ are conjugate to each other in the $N$-braid group $B_N$. 
In particular, 
$(e_{\bm{\omega}} o_{\bm{\omega}})^n$ is conjugate to 
$(\alpha_{\bm{\omega}})^n$ in $B_N$ 
for each  $n \ge 1$. 
\end{lemma}

\begin{proof}
In the same manner as in Example~\ref{ex_N7}, 
one can show that 
 $ e_{\bm{\omega}} o_{\bm{\omega}}$ is conjugate to  
 $\alpha_{\bm{\omega}}$  in $B_N$. 
 Take an $N$-braid $h$ so that 
 $ e_{\bm{\omega}}o_{\bm{\omega}} = h \alpha_{\bm{\omega}} h^{-1}$. 
 Then we have  
 $( e_{\bm{\omega}} o_{\bm{\omega}})^n = (h \alpha_{\bm{\omega}} h^{-1})^n = h (\alpha_{\bm{\omega}})^n h^{-1}$. 
 This completes the proof.
\end{proof}

We define a map (anti-homomorphism) 
\begin{eqnarray*}
\mathrm{rev}: B_n &\rightarrow& B_n
\\
\sigma_{i_1}^{\mu_1} \sigma_{i_2}^{\mu_2} \cdots \sigma_{i_k}^{\mu_k} 
&\mapsto& 
\sigma_{i_k}^{\mu_k} \cdots \sigma_{i_2}^{\mu_2} \sigma_{i_1}^{\mu_1}, 
\hspace{4mm} 
\mu_j= \pm 1.
\end{eqnarray*}
The image of 
$\alpha_{\bm{\omega}} = 
\sigma_1^{\omega_1} \sigma_2^{\omega_2} \cdots \sigma_{N-1}^{\omega_{N-1}}$ 
under the map $\mathrm{rev}$ is given by 
$$\mathrm{rev}(\alpha_{\bm{\omega}}) =
\sigma_{N-1}^{\omega_{N-1}}  \cdots \sigma_2^{\omega_2} \sigma_1^{\omega_1}.$$

\begin{lemma}
\label{lem_reverse}
For each $\bm{\omega} \in \Omega_N$, 
 $\alpha_{\bm{\omega}}$ is conjugate to 
$\mathrm{rev}(\alpha_{\bm{\omega}})$ in $B_N$. 
\end{lemma}

\begin{proof}
Suppose that $N=3$. 
For each $\bm{\omega} \in \Omega_3$,
we have $\alpha_{\bm{\omega}}= \sigma_1^{\omega_1} \sigma_2^{\omega_2}$ 
and 
$\mathrm{rev}(\alpha_{\bm{\omega}})= \sigma_2^{\omega_2} \sigma_1^{\omega_1}$. 
Clearly, the statement follows in this case. 

Suppose that $N=4$. 
For each   $\bm{\omega} \in \Omega_4$, 
we have 
$\alpha_{\bm{\omega}}= 
\sigma_1^{\omega_1} \sigma_2^{\omega_2} \sigma_3^{\omega_3}$ 
and 
$\mathrm{rev}(\alpha_{\bm{\omega}})= 
\sigma_3^{\omega_3} \sigma_2^{\omega_2} \sigma_1^{\omega_1}$. 
First, we consider the braid 
$b_0:= \sigma_3^{- \omega_3} 
(\mathrm{rev}(\alpha_{\bm{\omega}})) 
\sigma_3^{\omega_3} $. 
Then we have 
$$b_0= \sigma_3^{- \omega_3} 
(\sigma_3^{\omega_3} \sigma_2^{\omega_2} \sigma_1^{\omega_1}) 
\sigma_3^{\omega_3} 
= \sigma_2^{\omega_2} \sigma_1^{\omega_1} \sigma_3^{\omega_3} 
=  \sigma_2^{\omega_2} \sigma_3^{\omega_3}\sigma_1^{\omega_1}. 
$$
Next, we consider the braid 
$b_1:=(\sigma_2^{\omega_2} \sigma_3^{\omega_3})^{-1} b_0 
(\sigma_2^{\omega_2} \sigma_3^{\omega_3}) $. 
It is  written by 
$b_1= \sigma_1^{\omega_1} \sigma_2^{\omega_2} \sigma_3^{\omega_3} 
= \alpha_{\bm{\omega}}$. 
Hence $\mathrm{rev}(\alpha_{\bm{\omega}})$, $b_0$ and $\alpha_{\bm{\omega}}$ 
are conjugate to each other in $B_4$. 

In the same argument as above, 
one can verify the statement  of general $N$. 
We leave the rest of the proof to the reader. 
\end{proof}

We turn to the definition of the braid $\beta_{\bm{m}}$.  
Let $\bm{m}= (m_1, \dots, m_{k+1})$ 
be a $(k+1)$-tuple of positive integers with $k \ge 0$. 
Let $\beta_{\bm{m}}= \beta_{(m_1, \dots, m_{k+1})}$ denote the braid  
with $(1+ \sum_{i=1}^{k+1} m_i)$ strands 
as in Figure~\ref{fig_general_braid}. 
If we set $N:= 1+ \sum_{i=1}^{k+1} m_i$, then 
$\sum_{i=1}^{k+1} m_i= N-1$ and 
$\bm{m}$ is a composition of the integer $N-1$. 

Observe that if $k=0$ and $\bm{m}= (N-1) \in \Psi_{N-1}$, then 
$$\beta_{\bm{m}} = \beta_{(N-1)} = \sigma_1 \sigma_2 \cdots \sigma_{N-1} .$$
Here are some examples  written by Artin generators:   
$\beta_{(3,2)}= \sigma_1 \sigma_2 \sigma_3 \sigma_4^{-1} \sigma_5^{-1} 
\in B_6$,  
$\beta_{(1,1,1,1,1)}= \sigma_1 \sigma_2^{-1} \sigma_3 \sigma_4^{-1} \sigma_5 
\in B_6$. 
See Figure~\ref{fig_general_braid}. 

\begin{lemma}
\label{lem_periodic-braid}
If $\bm{m}= (N-1) \in \Psi_{N-1}$, then 
$\beta_{\bm{m}}$ is a periodic braid. 
\end{lemma}

\begin{proof}
If $\bm{m}= (N-1)$, then 
$(\beta_{\bm{m}})^N=(\sigma_1 \sigma_2 \cdots \sigma_{N-1})^N $ 
equals the full twist $\Delta^2$ corresponding to  the identity element in the group $B_N/Z(B_N)$. 
Hence $\beta_{\bm{m}}$ is a periodic braid. 
\end{proof}

\begin{figure}[htbp]
\begin{center}
\includegraphics[width=4.7in]{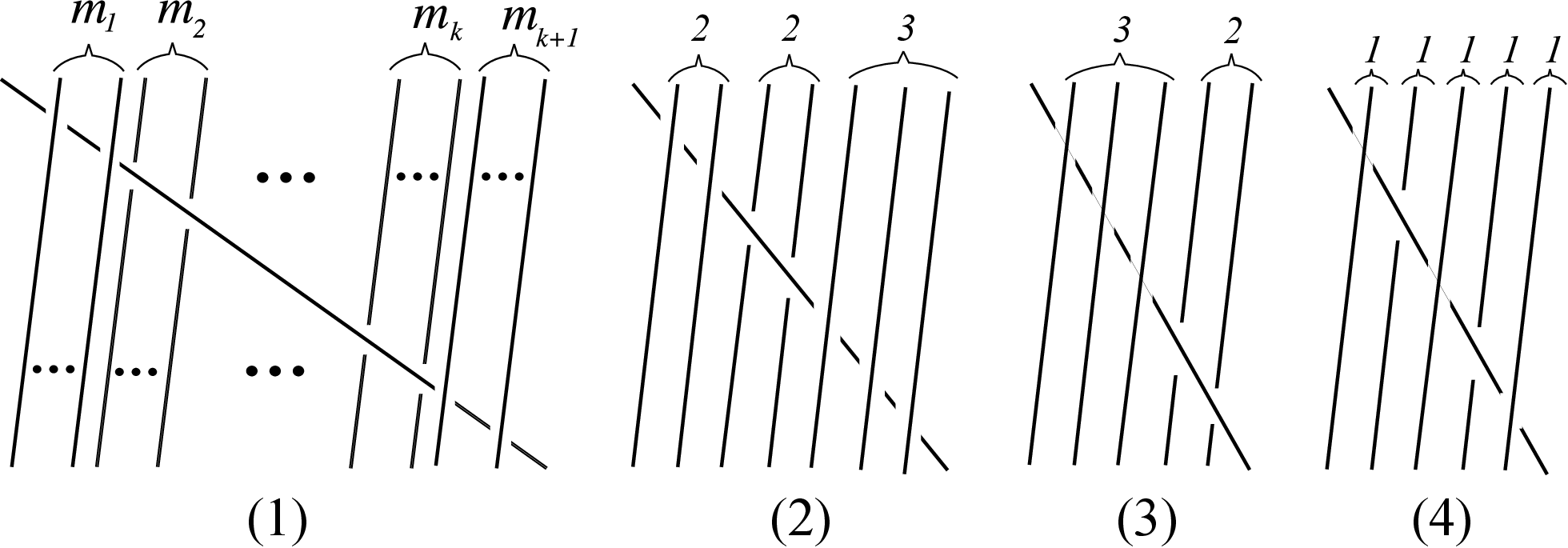}
\caption{(1) $\beta_{\bm{m}}=\beta_{(m_1, \dots, m_{k+1})}$. 
(2) $\beta_{(2,2,3)}$. (3) $\beta_{(3,2)}$. 
(4) $\beta_{(1,1,1,1,1)}$.}
\label{fig_general_braid}
\end{center}\
\end{figure}

\subsection{Stretch factors of pseudo-Anosov braids}
\label{subsection_strech-factors}
In this section, we study the stretch factor of the braid $\beta_{\bm{m}}$ 
when it is of pseudo-Anosov type. 
To do this, let $f(t)$ be an integral polynomial of degree $d$. 
The {\it reciprocal} of $f(t)$, denoted by $f_*(t)$, is defined by  
\begin{equation}
\label{equation_reciprocal} 
f_*(t)= t^d f(\tfrac{1}{t}).
\end{equation}
The following result gives a recursive formula for 
the stretch factor of $\beta_{\bm{m}}$. 

\begin{theorem}
\label{thm_recursive-example}
Let $\bm{m}= (m_1, \dots, m_{k+1})$ 
be a $(k+1)$-tuple of positive integers. 
If $k >0$, then 
the braid $\beta_{\bm{m}}$ is pseudo-Anosov. 
The stretch factor $\lambda_{\bm{m}}= \lambda_{(m_1, \dots, m_{k+1})}$ 
of $\beta_{\bm{m}}$ 
is the largest real root of the polynomial 
$$ F_{\bm{m}}(t)=F_{(m_1, \dots, m_{k+1})}(t):= 
t^{m_{k+1}} R_{(m_1, \dots, m_k)}(t)+ (-1)^{k+1} {R_{(m_1, \dots, m_k)}}_*(t),$$
where $R_{(m_1, \dots, m_i)}(t)$ is defined recursively as follows: 
\begin{eqnarray*}
R_{(m_1)}(t) &:=& t^{m_1+1} (t-1) -2t, \\
R_{(m_1, \dots, m_i)}(t)&:=& 
t^{m_i} (t-1) R_{(m_1, \dots, m_{i-1})}(t)+ 
(-1)^i 2t {R_{(m_1, \dots, m_{i-1})}}_*(t). 
\end{eqnarray*}
 \end{theorem}

See Proposition 4.1, Theorem 1.2 in \cite{KinTakasawa08} 
for the proof of Theorem~\ref{thm_recursive-example}. 
The next lemma explains a relation between  braids 
$\alpha_{\bm{\omega}}$ (see (\ref{equation_alpha})) and $\beta_{\bm{m}}$.

\begin{lemma}
\label{lem_alpha-beta}
Let $\Theta: \Psi_{N-1} \rightarrow \Omega_N^+$ be the bijection 
as in Lemma~\ref{lem_bijection}. 
Then the identity  
$\beta_{\bm{m}}= \alpha_{\Theta(\bm{m})}$ holds 
for each $\bm{m} \in \Psi_{N-1}$. 
Moreover, $\beta_{\bm{m}}$ is pseudo-Anosov 
if and only if $\bm{m} \ne (N-1) $. 
\end{lemma}

\begin{proof}
The former statement is immediate by the definitions of 
$\alpha_{\bm{\omega}}$ and $\beta_{\bm{m}}$. 
For the latter statement, 
we represent a composition 
$\bm{m} \in \Psi_{N-1}$ by a $(k+1)$-tuple $\bm{m}= (m_1, \dots, m_{k+1})$ 
of positive integers with $k \ge 0$ and $N-1= \sum_{i=1}^{k+1}m_i$. 
By Lemma~\ref{lem_periodic-braid} and Theorem~\ref{thm_recursive-example}, 
$\beta_{\bm{m}}$ is pseudo-Anosov if and only if $k>0$, equivalently 
$\bm{m} \ne (N-1)$. 
\end{proof}

For convenience of the readers, we explain how to compute the stretch factor 
$\lambda_{\bm{m}}$ of the pseudo-Anosov braid $\beta_{\bm{m}}$. 
For more details, see \cite{KinTakasawa08}.

\begin{proof}[How to compute  $\lambda_{\bm{m}}$]
We take a composition $\bm{m}= (m_1,  \dots, m_{k+1}) 
\in \Psi_{N-1}$ with $k>0$. 
Theorem~\ref{thm_recursive-example} tells us that 
$\mathfrak{c}(\Gamma(\beta_{\bm{m}})) \in \mathrm{MCG}(\Sigma_{0,N+1})$ 
is a pseudo-Anosov mapping class. 
We identify $\beta_{\bm{m}}$ with 
$\mathfrak{c}(\Gamma(\beta_{\bm{m}}))$. 
We view  an $N+1$-punctured sphere 
$\Sigma_{0, N+1}$ as a sphere with $m_i+1$ marked points $X_i$ 
circling an unmarked point  $u_i$ for each $i \in \{1, \dots, k+1\}$ 
and a single marked point $\infty$ (corresponding to $\partial D$). 
Note that 
$|X_i|= 1+ m_i$ for each $i \in \{1, \dots, k+1\}$ and 
$|X_j \cap X_{j+1}|= 1$ for each $j \in \{1, \dots, k\}$. 
See Figure~\ref{fig_marked-point}(1).

We choose a finite graph
$G_{\bm{m}} \subset \Sigma_{0, N+1}$ that is homotopy equivalent to
the $N+1$-punctured sphere. 
The graph $G_{\bm{m}}$ has $N$ loop edges, 
each of which encircles a puncture (a marked point), 
and $N+k$ non-loop edges. 
See Figure~\ref{fig_marked-point}(2). 
Let $P$ be the set of $N$ loop edges of $G_{\bm{m}}$. 
The graph $G_{\bm{m}}$ has $N+k+1$ vertices. 
For each loop edge, there is a vertex of degree $3$ or $4$. 
The unmarked point $u_i$ corresponds to a vertex of degree $1+ m_i$ 
for $i \in \{1, \dots, k+1\}$. 

\begin{figure}[htbp]
\begin{center}
\includegraphics[width=3.7in]{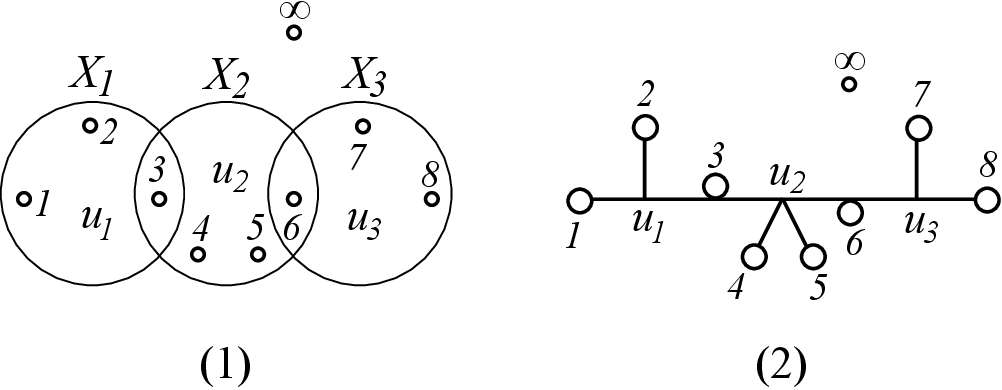}
\caption{Case $N=8$, $k=2$ and $\bm{m}= (m_1, m_2, m_3)=
(2,3,2) \in \Psi_{7}$. 
(1) $N+1$-marked points in the sphere. 
(Small circles indicate marked points (punctures).) 
(2) The graph $G_{\bm{m}}$. 
(Each loop edge encircles a marked point. 
In this case, $G_{\bm{m}}$ has $8$ loop edges and $10$ non-loop edges.)}
\label{fig_marked-point}
\end{center}
\end{figure}

Given a mapping class $\psi \in \mathrm{MCG}(\Sigma_{0,N+1})$, 
one can pick an induced graph map 
$g: G_{\bm{m}} \rightarrow G_{\bm{m}}$
(see \cite[Section~1]{BestvinaHandel95}). 
We require that $g$ sends vertices to vertices, edges to edge paths and 
satisfies $g(P)= P$. 
We may suppose that $g$ has \textit{no backtracks}, i.e., 
$g$ maps each oriented edge of $G_{\bm{m}} $ to an edge path which does not contain an oriented edge $e$ 
followed by the reverse edge $\overline{e}$ of 
the oriented edge $e$. 
Then the graph map $g$ defines an $N+k$ by $N+k$ \textit{transition matrix} 
$M$ with respect to the $N+k$ non-loop edges. 
More precisely, 
for $r,s \in \{1, \dots, N+k\}$ 
the $rs$-entry $M_{rs}$ is the number of times that the $g$-image 
of the $s$th edge  runs the $r$th  edge in either direction. 
We say that 
$g: G_{\bm{m}} \rightarrow G_{\bm{m}}$ is  \textit{efficient} 
if $g^n: G_{\bm{m}} \rightarrow G_{\bm{m}}$ has no backtracks for all $n  > 0$. 

We now define  $\phi_{\bm{m}} \in \mathrm{MCG}(\Sigma_{0, N+1})$. 
We will see in (\ref{equation_claim}) that 
$\beta_{\bm{m}}$ and $\phi_{\bm{m}}$ are conjugate in 
$\mathrm{MCG}(\Sigma_{0, N+1})$. 
Let $f_i= f_{\bm{m}, i}: \Sigma_{0, N+1} \rightarrow \Sigma_{0, N+1}$ 
be a homeomorphism such that 
$f_i$ rotates the marked points of $X_i$ counterclockwise around $u_i$ 
if $i$ is odd, 
and $f_i$ rotates the marked points of $X_i$ clockwise around $u_i$ 
if $i$ is even. 
See Figure~\ref{fig_graph-map}(1)(2). 
Define $\phi_{i} = [f_i] \in \mathrm{MCG}(\Sigma_{0, N+1})$. 
Figure~\ref{fig_graph-map}(3)(4) 
illustrates the $N$-braid $b_i= b_{\bm{m}, i}$ 
such that $\phi_i= \mathfrak{c}(\Gamma(b_i))$.  
We set 
$\phi_{\bm{m}}= \phi_{k+1} \circ \cdots \circ \phi_2 \circ \phi_1 $. 
(c.f. \cite[Figure~3]{KinTakasawa08} for $\phi_{\bm{m}}= \phi_{(4,2,1)}$.) 
Recall that 
$\mathrm{rev}: B_n \rightarrow B_n$ is the map as in Section~\ref{subsection_pAfamily}. 
By definitions of  $f_i$ and  $\beta_{\bm{m}}$ 
(see Figure~\ref{fig_general_braid}), 
we can verify that $\phi_{\bm{m}} $ is of the form 
$$\phi_{\bm{m}} 
= \mathrm{rev}(\beta_{\bm{m}}) \in \mathrm{MCG}(\Sigma_{0, N+1}).$$
For example, 
when $N=5$, $k=1$ and $\bm{m}= (2,2) \in \Phi_{4}$, 
we have 
$$\phi_{\bm{m}}= \phi_2 \circ \phi_1 
= \sigma_4^{-1} \sigma_3^{-1} \cdot \sigma_2 \sigma_1 
= \mathrm{rev}(\sigma_1 \sigma_2 \sigma_3^{-1} \sigma_4^{-1}) = \mathrm{rev}(\beta_{(2,2)}).$$
Let $\Theta: \Psi_{N-1} \rightarrow \Omega_N^+$ be the bijection 
as in Lemma~\ref{lem_bijection}. 
Then $\beta_{\bm{m}}= \alpha_{\Theta(\bm{m})}$ by Lemma~\ref{lem_alpha-beta}. 
Moreover $\alpha_{\Theta(\bm{m})}$ and $\mathrm{rev}(\alpha_{\Theta(\bm{m})})$ 
are conjugate  in $B_N$ by Lemma~\ref{lem_reverse}. 
As a consequence, one sees that 
\begin{equation}
\label{equation_claim}
\phi_{\bm{m}} (= \mathrm{rev}(\beta_{\bm{m}}) = 
\mathrm{rev}(\alpha_{\Theta(\bm{m})}))\ 
\mbox{and}\ 
\beta_{\bm{m}}  \ 
\mbox{are conjugate in}\ \mathrm{MCG}(\Sigma_{0, N+1}). 
\end{equation}

\begin{figure}[htbp]
\begin{center}
\includegraphics[width=4.2in]{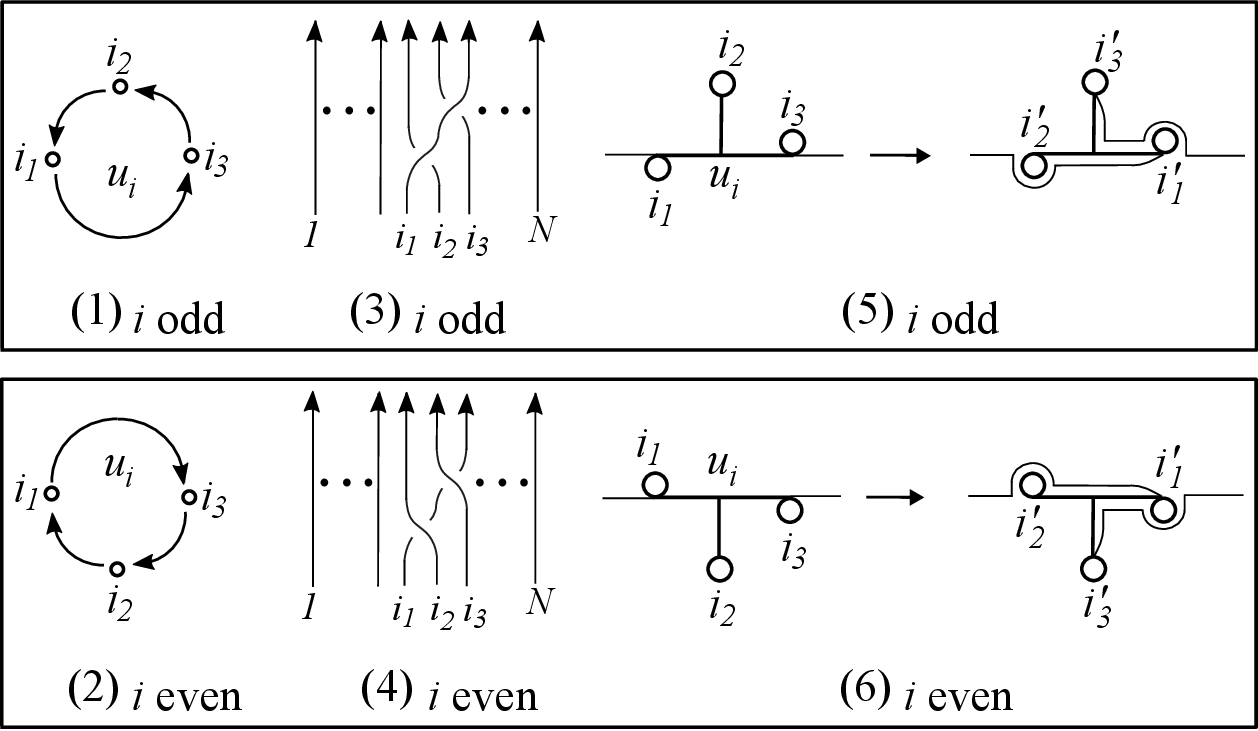}
\caption{
Case $m_i=2$. 
(1)(2) 
$f_i: \Sigma_{0, N+1} \rightarrow \Sigma_{0, N+1}$ 
when $i$ is odd/even. 
(3)(4) The braid $b_i$ corresponding to $\phi_i$ 
when $i$ is odd/even. 
(5)(6) $g_i: G_{\bm{m}} \rightarrow G_{\bm{m}}$ 
when $i$ is odd/even.}
\label{fig_graph-map}
\end{center}
\end{figure}

The mapping class $\phi_i= [f_i]$ for $i \in \{1, \dots, k+1\}$ 
induces a graph map $g_i= g_{\bm{m},i}: G_{\bm{m}} \rightarrow G_{\bm{m}}$ 
which has no backtracks as shown in Figure~\ref{fig_graph-map}(5)(6). 
We denote by $M_i = M_{\bm{m}, i}$, the transition matrix of $g_i$ 
(with respect to the $N+k$ non-loop edges). 
The composition 
\begin{equation}
\label{equation_graphmap}
g_{\bm{m}}:= g_{k+1} \circ \cdots \circ g_2 \circ g_1: 
G_{\bm{m}} \rightarrow G_{\bm{m}}
\end{equation}
\if0
$g_{\bm{m}}:= g_{k+1} \circ \cdots \circ g_2 \circ g_1: 
G_{\bm{m}} \rightarrow G_{\bm{m}}$ 
\fi
is an induced graph map of  $\phi_{\bm{m}}$. 
By the induction on $k$, one can show that 
$g_{\bm{m}}:G_{\bm{m}} \rightarrow G_{\bm{m}} $ has no backtracks. 
This implies that 
the transition matrix of $g_{\bm{m}}$ 
with respect to the non-loop edges is given by 
$$M_{\bm{m}}:= M_{k+1}  \cdots \cdot M_2 \cdot M_1.$$
It is proved in \cite{KinTakasawa08} that 
$M_{\bm{m}}$ is a Perron-Frobenius matrix. 
\if0
(Here a square matrix $M$ with nonnegative integer entries 
is \textit{Perron-Frobenius} 
if some power of $M$ is a positive matrix.) 
By the Perron-Frobenius theorem \cite[Theorem 1.1]{Seneta06}, 
$M_{\bm{m}}$ has a real eigenvalue 
$\lambda(M_{\bm{m}})>1$ which exceeds the moduli of all other eigenvalues. 
\fi
Moreover one can show that 
$g_{\bm{m}}:G_{\bm{m}} \rightarrow G_{\bm{m}}$ is efficient. 
As a consequence of  Bestvina-Handel algorithm \cite{BestvinaHandel95}, 
the Perron-Frobenius eigenvalue 
$\lambda(M_{\bm{m}})$ of $M_{\bm{m}}$ 
equals the stretch factor $\lambda(\phi_{\bm{m}})$ of $\phi_{\bm{m}}$. 
By (\ref{equation_claim}), $\phi_{\bm{m}}$ is conjugate to $\beta_{\bm{m}}$ 
in $\mathrm{MCG}(\Sigma_{0, N+1})$. 
We obtain 
\begin{equation}
\label{equation_how-to-compute}
\lambda_{\bm{m}} (= \lambda(\beta_ {\bm{m}}))= \lambda(M_{\bm{m}}). 
\end{equation}
The Perron-Frobenius eigenvalue $\lambda(M_{\bm{m}})$ 
gives the stretch factor $\lambda_{\bm{m}}$. 
\end{proof}

Fixing  $N \ge 3$, 
let $Y_N$ be the set of braids 
$\beta_{\bm{m}} $ over all  $\bm{m} \in \Psi_{N-1}$. 
$$Y_N:= \{ \beta_{\bm{m}}\  | \ 
\bm{m}=(m_1, \dots, m_{k+1}) \in \Psi_{N-1}\ \mbox{with }k \ge 0 \} \subset B_N.$$
By Lemma~\ref{lem_alpha-beta}, 
all braids $\beta_{\bm{m}}$ in $Y_N$ except the case $\bm{m}= (N-1)$ 
 are pseudo-Anosov. 
By Lemma~\ref{lem_bijection}, $Y_N$ can be written by 
\begin{equation}
\label{equation_Y}
Y_N= \{\alpha_{\bm{\omega}}\ |\ \bm{\omega} \in \Omega_N^+\}.
\end{equation}
\if0
The following is an example of the set $Y_N$ when $N= 3,4$ and $5$.  
\begin{eqnarray*}
Y_3 &=& \{\beta_{(2)}, \beta_{(1,1)}\}, 
\\
Y_4 &=& \{\beta_{(3)}, \beta_{(1,2)}, \beta_{(2,1)}, \beta_{(1,1,1)}\}, 
\\
Y_5 &=& \{\beta_{(4)}, \beta_{(1,3)}, \beta_{(2,2)}, \beta_{(3,1)}, \beta_{(2,1,1)}, 
\beta_{(1,2,1)}, \beta_{(1,1,2)}, \beta_{(1,1,1,1)}\}.
\end{eqnarray*}
\fi
The following result identifies the braids in $Y_N$ 
with the largest and smallest stretch factor, respectively.

\begin{theorem}
\label{thm_max-mini}
Among all braids in $Y_N \setminus \{\beta_{(N-1)}\}$, 
the braid $\beta_{\bm{1}_{N-1}}= \beta_{(1, 1, \dots, 1)} $  realizes 
the largest stretch factor, while the smallest stretch factor is realized by  $\beta_{(n, n)} $ if $ N = 2n+1$, and by $\beta_{(n-1, n)}$ 
(and $\beta_{(n, n-1)}$) if $ N = 2n $.
\end{theorem}

To prove Theorem~\ref{thm_max-mini}, 
we need the following results.

\begin{proposition}[Proposition~1.1 in \cite{KinTakasawa08}]
\label{prop_decreasing-example}
For $k >0$, 
we consider $(k+1)$-tuples of positive integers 
$\bm{m}= (m_1, \dots, m_{k+1})$ and 
$\bm{m}'= (m_1', \dots, m_{k+1}')$. 
Suppose that $m_i'= m_i+1$ for some $i$ 
and $m_j'= m_j$ if $j \ne i$. 
Then we have $\lambda_{\bm{m}'} < \lambda_{\bm{m}}$.  
\end{proposition}

Using Proposition~\ref{prop_decreasing-example} 
repeatedly, we obtain the following.

\begin{corollary}
\label{cor_decreasing-example}
For $k >0$, we consider $(k+1)$-tuples of positive integers 
$\bm{m}= (m_1, \dots, m_{k+1})$ and 
$\bm{m}'= (m_1', \dots, m_{k+1}')$. 
Suppose that 
$m_j \le m_j'$ for all $j$. 
Then we have $\lambda_{\bm{m}'} \le  \lambda_{\bm{m}}$. 
The equality holds  if and only if  $\bm{m}=\bm{m}' $, i.e., 
$m_j= m_j'$ for all $j$. 
\end{corollary}

\begin{proposition}
\label{prop_increasing-example}
For $k >0$, 
we consider 
a $(k+1)$-tuple of positive integers 
 $\bm{m}= (m_1, \dots, m_{k+1})$ and 
a $(k+2)$-tuple of positive integers 
$\bm{m}'= (m_1', \dots, m_{k+1}', m_{k+2}')$. 
We assume that 
$\bm{m}'$ is of the form 
$$\bm{m}'= (m_1, \dots, m_{k+1}, m_{k+2}'),$$
i.e., $m_j= m_j'$ for all $j= 1, \dots, k+1$. 
Then we have 
$\lambda_{\bm{m}} < \lambda_{\bm{m}'}$. 
\end{proposition}

\begin{proof}
Consider the  graph maps 
$g_{\bm{m}}: G_{\bm{m}} \rightarrow G_{\bm{m}}$  and 
$g_{\bm{m}'}: G_{\bm{m}'} \rightarrow G_{\bm{m}'}$ 
as in (\ref{equation_graphmap}). 
Note that $G_{\bm{m}}$ is a subgraph of $G_{\bm{m}'}$ 
by the assumption on $\bm{m}'$. 
Let $M_{\bm{m}}$ and $M_{\bm{m}'}$ be the transtion matrices of 
$g_{\bm{m}}$ and $g_{\bm{m}'}$, respectively. 
Let $\lambda(M_{\bm{m}})$ and $\lambda(M_{\bm{m}'})$ be the 
corresponding Perron-Frobenius eigenvalues. 
By (\ref{equation_how-to-compute}), 
it is enough to prove that 
$ \lambda(M_{\bm{m}})< \lambda(M_{\bm{m}'})$. 
Under the suitable labeling of non-loop edges of $G_{\bm{m}'}$, 
the matrix $M_{\bm{m}'}$ can be written by 
$$M_{\bm{m}'}= 
\left[\begin{array}{cc}M_{\bm{m}}  & A \\B & C\end{array}\right],
$$
where $A,B$ and $C$ are block matrices with nonnegative integer entries 
and $C$ is a square matrix. 
Assume that $A= \bm{0}$ and $B= \bm{0}$. 
Then any power $M_{\bm{m}'}^k$ of $M_{\bm{m}'}$ is not a positive matrix, which contradicts the fact that $M_{\bm{m}'}$ is Perron-Frobenius. 
Hence,  either $A$ or $B$ is a non-zero matrix. 
Consider the square matrix $D$ with the same size as $M_{\bm{m}'}$ of the form 
$D= 
\left[\begin{array}{cc}M_{\bm{m}}  & \bm{0} \\\bm{0} & \bm{0}\end{array}\right]$. 
Then  $D \ne M_{\bm{m}'}$ 
since either $A$ or $B$ is a non-zero matrix. 
Since $A, B, C \ge \bm{0}$, we have 
$\bm{0} \le D \le M_{\bm{m}'}$, 
i.e., $0 \le D_{st} \le (M_{\bm{m}'})_{st}$ for each $st$-entry. 
Then the Perron-Frobenius theorem (\cite[Theorem 1.1(e)]{Seneta06})
tells us that 
if $\lambda$ is an eigenvalue of $D$, then 
$|\lambda| < \lambda(M_{\bm{m}'})$. 
Thus, we obtain 
$ \lambda(M_{\bm{m}})< \lambda(M_{\bm{m}'})$. 
This completes the proof. 
\end{proof}

\begin{proof}[Proof of Theorem~\ref{thm_max-mini}]
We consider the braid $\beta_{\bm{m}}$ associated with a $(k+1)$-tuple 
$\bm{m}= (m_1, \dots, m_{k+1})$. 
Suppose that $\bm{m}$ is a composition of $N-1$. 
Clearly, $k+1 \le N-1$ and 
$\beta_{\bm{m}}= \beta_{(m_1, \dots, m_{k+1})} \in Y_N$. 
By Corollary~\ref{cor_decreasing-example}, we have the inequality 
\begin{equation}
\label{equation_decreasing-innequality}
  \lambda(\beta_{(m_1, \dots, m_{k+1})}) \le 
\lambda(\beta_{\bm{1}_{k+1}})\hspace{5mm} 
\mbox{for all\ integers } m_1, \dots, m_{k+1} \ge 1.
\end{equation}
The equality holds if and only if $\bm{m}= \bm{1}_{k+1}$. 
The inequalities $k+1 \le N-1$ and 
(\ref{equation_decreasing-innequality}) 
together with Proposition~\ref{prop_increasing-example} 
tell us that 
\if0
for each $m \ge 1$ we have 
\begin{equation}
\label{equation_two-inequality}
\lambda(\beta_{\bm{1}_{k+1}}) < 
\lambda(\beta_{(\underbrace{1,\dots, 1}_{k+1}, m)}) \le 
\lambda(\beta_{\bm{1}_{k+2}}). 
\end{equation}
The first inequality in (\ref{equation_two-inequality}) 
is given by Proposition~\ref{prop_increasing-example}.  
The second inequality in (\ref{equation_two-inequality})  comes from 
Corollary~\ref{cor_decreasing-example}, and 
the equality holds if and only if $m=1$. 
Putting (\ref{equation_decreasing-innequality}) 
and (\ref{equation_two-inequality}) together, we obtain 
\fi
$$ \lambda(\beta_{(m_1, \dots, m_{k+1})}) \le 
\lambda(\beta_{\bm{1}_{k+1}})  \le 
\lambda(\beta_{\bm{1}_{N-1}}).$$ 
Thus, 
$\beta_{\bm{1}_{N-1}} \in Y_N$ 
realizes the largest stretch factor. 

Next, we turn to the braid in $Y_N$ with the smallest stretch factor. 
Let  $\beta_{(m_1, \dots, m_{k+1}, m)}$ 
be the  braid associated with a $(k+2)$-tuple 
$(m_1, \dots, m_{k+1}, m)$. 
Suppose  this $(k+2)$-tuple is a composition of $N-1$. 
Then 
$\beta_{(m_1, \dots, m_{k+1}, m)} \in Y_N$. 
Note that the braid $\beta_{(m_1, \dots, m_k, m_{k+1}+m)}$ 
associated with the $(k+1)$-tuple 
$(m_1, \dots, m_k, m_{k+1}+m) \in \Psi_{N-1}$  
is also an element of  $Y_N$. 
Proposition~\ref{prop_increasing-example} tells us that 
\begin{equation}
\label{equation_first-innequality}
\lambda(\beta_{(m_1, \dots, m_k, m_{k+1})} ) 
< \lambda(\beta_{(m_1, \dots, m_k, m_{k+1}, m)}).
\end{equation}
By Corollary~\ref{cor_decreasing-example}, 
we have 
\begin{equation}
\label{equation_second-innequality}
\lambda(\beta_{(m_1, \dots, m_k, m_{k+1}+m)}) 
< \lambda(\beta_{(m_1, \dots, m_k, m_{k+1})}).
\end{equation}
By (\ref{equation_first-innequality}) and (\ref{equation_second-innequality}), 
the stretch factor of 
$\beta_{(m_1, \dots, m_k, m_{k+1}+m)} \in Y_N$  
is smaller than that of $\beta_{(m_1, \dots, m_k, m_{k+1}, m)} \in Y_N$. 
This means that for the braid with the smallest stretch factor, 
it is enough to consider elements $\beta_{\bm{m}} \in Y_N$ 
associated with the compositions $\bm{m} \in \Psi_{N-1}$  of the form $\bm{m}= (m,n)$. 
The inverse  $\beta_{(m,n)}^{-1}$ of $\beta_{(m,n)}$ 
satisfies 
$\Delta \beta_{(m,n)}^{-1} \Delta^{-1}=\beta_{(n,m)} $, 
where $\Delta$ is the half twist. 
Hence $\beta_{(m,n)}^{-1}$ and $\beta_{(n,m)}$ are conjugate 
in $B_N$. 
In particular, we have $\lambda(\beta_{(m,n)}^{-1}) = \lambda(\beta_{(n,m)})$. 
Since a pseudo-Anosov braid $b$ and its inverse $b^{-1}$ have the same 
stretch factor, we conclude that 
$$\lambda(\beta_{(m,n)}) = \lambda(\beta_{(m,n)}^{-1})=
\lambda(\beta_{(n,m)}).$$
(The equality $\lambda(\beta_{(m,n)})= \lambda(\beta_{(n,m)})$ 
also follows from Lemma~\ref{lem_stretchfactor-equivalent}.)
Hence, we restrict our attention to the pairs 
$(m,n)$ with $m \le n$. 
The following inequalities are proved in \cite[Proposition~3.33]{HironakaKin06}. 
\begin{eqnarray*}
\lambda(\beta_{(n,n)}) &<& \lambda(\beta_{(n-k, n+k)}) 
\hspace{7mm} \mbox{for}\ k= 1,2,  \dots, n-1, 
\\
\lambda(\beta_{(n-1,n)}) &<& \lambda(\beta_{(n-k-1, n+k)}) 
\hspace{3mm} \mbox{for}\ k= 1,2,  \dots, n-2. 
\end{eqnarray*}
Thus, 
if $ N = 2n+1$ (resp. $N= 2n$), then 
the smallest stretch factor is realized by  $\beta_{(n, n)} $ 
(resp. $\beta_{(n-1, n)}$ and $\beta_{(n, n-1)}$).
This completes the proof. 
\end{proof} 

By Theorem~\ref{thm_max-mini}, 
we are interested in the computation of the stretch factors of 
$\beta_{\bm{1}_{N-1}}$ and $\beta_{(m,n)}$ with $|m-n|= 0$ or $1$. 
Examples~\ref{ex_111} and \ref{ex_beta-mn} are useful.

\begin{example}
\label{ex_111}
Let us compute 
the stretch factors of $\beta_{\bm{1}_3}$, 
$\beta_{\bm{1}_4}$ and 
$\beta_{\bm{1}_5}$. 
\if0
Let us compute 
the stretch factors of $\beta_{\bm{1}_3}= \beta_{(1,1,1)}$, 
$\beta_{\bm{1}_4}=\beta_{(1,1,1,1)}$ and 
$\beta_{\bm{1}_5}=\beta_{(1,1,1,1,1)}$. 
\fi
\begin{enumerate}
\item 
By recursive formulas of $R_{\bm{m}}(t)$ and $F_{\bm{m}}(t)$ 
(see Theorem~\ref{thm_recursive-example}) 
and the definition of $f_*(t)$ (see (\ref{equation_reciprocal})), 
we obtain 
\begin{eqnarray*}
R_{(1)}(t)&=& t^3-t^2-2t, 
\\
R_{(1,1)}(t) &=& t(t-1)R_{(1)}(t)+ 2t {R_{(1)}}_*(t) 
= t^5-2t^4-5t^3+2t, 
\\
F_{(1,1,1)}(t)&=&  
t R_{(1,1)}(t) - {R_{(1,1)}}_*(t) 
= (t-1)(t+1)^3(t^2-4t+1). 
\end{eqnarray*}
Hence, the largest real root of the third factor $t^2-4t+1$ of 
$F_{(1,1,1)}(t)$ is equal to  $\lambda(\beta_{\bm{1}_3})= 2+ \sqrt{3}$.  

\item 
A computation shows that 
\begin{eqnarray*}
R_{(1,1,1)}(t)&=& t(t-1) R_{(1,1)}(t) - 2t {R_{(1,1)}}_*(t) 
\\
&=& 
t^7-3t^6-7t^5+5t^4+12t^3+2t^2-2t, 
\\
F_{(1,1,1,1)}(t)&=& t R_{(1,1,1)}(t)+ {R_{(1,1,1)}}_*(t) 
\\
&=& (t+1)^4 (t^4-7t^3+13t^2-7t+1). 
\end{eqnarray*}
The largest real root of the second factor of $F_{(1,1,1,1)}(t)$ 
gives the stretch factor 
 $\lambda(\beta_{\bm{1}_4}) \approx 4.39026$. 

\item 
Lastly, we compute 
\begin{eqnarray*}
R_{(1,1,1,1)}(t)&=& t^9-4t^8-8t^7+16t^6+31t^5-18t^3-4t^2+2t, 
\\
F_{(1,1,1,1,1)}(t)&=& t R_{(1,1,1,1)}(t) - {R_{(1,1,1,1)}}_*(t) 
\\
&=&(t-1)(t+1)^5(t^2-3t+1)(t^2-5t+1). 
\end{eqnarray*}
Hence, the largest real root of the last factor $t^2-5t+1$ 
of $F_{(1,1,1,1,1)}(t)$ 
gives us 
$\lambda(\beta_{\bm{1}_5}) \approx 4.79129$. 
\end{enumerate}
\end{example}

\begin{example}
\label{ex_beta-mn}
A computation shows that 
\begin{eqnarray*}
R_{(m)}(t)&=& t^{m+1}(t-1)-2t= t^{m+2}-t^{m+1}-2t, 
\\
{R_{(m)}}_*(t)&=&t^{m+2}R_{(m)}(\tfrac{1}{t})= 1-t-2t^{m+1}. 
\end{eqnarray*}
By Theorem~\ref{thm_recursive-example}, 
the stretch factor  of  $\beta_{(m,n)}$  
is the largest real root of 
$$F_{(m,n)}(t)= t^n R_{(m)}(t)+ {R_{(m)}}_*(t) 
= t^n (t^{m+2}-t^{m+1}-2t) -2t^{m+1}-t+1.$$
\end{example}

In Table~\ref{table_max-mini-stretch-factor}, 
we list the smallest and largest stretch factors 
among all pseudo-Anosov braids in $Y_N$.

\begin{table}[hbtp]
\caption{
Smallest and largest stretch factors 
in $Y_N \setminus \{\beta_{(N-1)}\}$.}
\label{table_max-mini-stretch-factor}
\begin{center}
\begin{tabular}{|c|c|c|}
\hline
$N$ & $\displaystyle\min_{\substack{ \beta \in Y_N \setminus \{\beta_{(N-1)}\}}} \lambda(\beta)$ & $\displaystyle\max_{\substack{ \beta \in Y_N \setminus \{\beta_{(N-1)}\}}} \lambda(\beta)$
\\
\hline 
$3$ & $\lambda(\beta_{\bm{1}_{2}})=  \tfrac{3+ \sqrt{5}}{2}$ & 
$\lambda(\beta_{\bm{1}_{2}})=\tfrac{3+ \sqrt{5}}{2} $
\\
\hline 
$4$ & $\lambda(\beta_{(1,2)})\approx 2.29663$ & 
$\lambda(\beta_{\bm{1}_{3}})=2+ \sqrt{3} $
\\
\hline
$5$ & $\lambda(\beta_{(2,2)}) \approx 2.01536 $ & 
$\lambda(\beta_{\bm{1}_{4}})\approx 4.39026 $
\\
\hline
$6$ & $\lambda(\beta_{(2,3)})\approx 1.8832  $ & 
$\lambda(\beta_{\bm{1}_{5}})\approx 4.79129 $
\\
\hline
$7$ & $\lambda(\beta_{(3,3)})\approx 1.75488  $ & 
$\lambda(\beta_{\bm{1}_{6}})\approx 5.04892 $
\\
\hline
$8$ & $\lambda(\beta_{(3,4)})\approx 1.6815  $ & 
$\lambda(\beta_{\bm{1}_{7}})\approx 5.22274 $
\\
\hline
$9$ & $\lambda(\beta_{(4,4)})\approx 1.60751  $ & 
$\lambda(\beta_{\bm{1}_{8}})\approx 5.345 $
\\
\hline
$10$ & $\lambda(\beta_{(4,5)})\approx 1.56028  $ & 
$\lambda(\beta_{\bm{1}_{9}})\approx 5.43401 $
\\
\hline
\end{tabular}
\end{center}
\end{table}

\section{Simple choreographies by Yu}
\label{section_Yu}

In this section, we explain simple choreographies of the planar $N$-body problem obtained by Yu \cite{Yu17}. 
His main theorem is:
\begin{theorem}[\cite{Yu17}]
\label{theorem:Yu}
\label{theorem_main-Yu}
    For every $N \ge 3$, there exist at least $2^{N-3} + 2^{\lfloor(N-3)/2\rfloor}$ different simple choreographies for the planar Newtonian $N$-body problem 
with equal masses,
    where $\lfloor \cdot \rfloor$ denotes the integer part of a real number.
\end{theorem}

In the end of this section, we will explain the meaning of 
``different simple choreographies" in the statement of Theorem~\ref{theorem:Yu}.

We identify the plane $\R^2$ with the complex plane $\C$.   
The planar \(N\)-body problem with equal masses is described by the following differential equation:
\begin{align}
\label{eq:nbp}
\ddot{z}_{j} = \sum_{k \in \{0,1,\dots,N-1\} \backslash \{j\}}-\frac{z_j-z_k}{|z_j-z_k|^3} \quad (j \in \{0,1,\dots,N-1\})
\end{align}
where $\bm{z}=(z_j)_{j=0}^{N-1} \in \mathbb{C}^{N}$. 
%where $\bm{z}=(z_j)_{j \in \{0, 1, \dots,N-1\}} \in \mathbb{C}^{N}$.
The \(N\)-body problem has a variational structure. That is, the critical points of the functional
\[
\mathcal{A}_{[a,b]}(\bm{z})
= \int_{a}^{b} L(\bm{z},\dot{\bm{z}}) dt
\]
correspond to weak solutions of the $N$-body problem,
where
\[
\bm{z} \in H^1([a,b], \mathbb{C}^{N})
:=\{ \bm{x} \colon [a,b] \to \mathbb{C}^{N} \mid \bm{x}, \dot{\bm{x}} \in L^2([a,b], \mathbb{C}^{N})\}
\]
and
\[
L(\bm{z},\dot{\bm{z}})= \frac{1}{2}\sum_{j=0}^{N-1} |\dot{z}_j|^2 + \sum_{\substack{j,k \in \{0,1,\dots,N-1\}, \\ j<k}} \frac{1}{|z_j-z_k|}.
\]

Theorem \ref{theorem:Yu} was proved using this variational structure.  
More precisely, he showed the existence of a minimizer of $\mathcal{A}_{[0,N]}(\bm{z})$ in the $N$-periodic functional space
\[
\Lambda_N:={H^1(\R / N \mathbb{Z}},\mathbb{C}^{N})
\]
under symmetric and topological constraints.

Firstly, we explain the symmetric condition.
Let $G$ be a finite group and define three actions $\tau$, $\rho$ and $\sigma$ as follows:
\begin{align*}
    &\tau \colon G \to O(2), \ \text{(the action of \( G \) on the time circle \( \mathbb{R} / 2\pi \mathbb{Z} \))}, \\
    &\rho \colon G \to O(2), \ \text{(the action of \( G \) on two-dimensional Euclidean space)}, \text{ and} \\
    &\sigma \colon G \to S_{N}, \ \text{(the action of \( G \) on the index set \( \{0,1,\dots,N-1\} \))},
\end{align*}
where \( O(2) \) and \( S_{N} \) represent the two-dimensional orthogonal group and the  symmetric group of $N$ elements, respectively. 
%and the \( N \)-dimensional symmetric group, respectively. 
For each $g \in G$, we define its action as follows:
\[
g(\bm{z}(t))=(\rho(g) z_{\sigma(g^{-1})(0)} (\tau(g^{-1})t), \dots,
\rho(g) z_{\sigma(g^{-1})(N-1)} (\tau(g^{-1})t).
\]
Set  
\begin{align*}
\Lambda_{N}^{G} &= \{ \bm{z} \in \Lambda_{N} \mid g(\bm{z}(t)) = \bm{z}(t) \ \text{for all} \ g \in G \},\\
\hat{\Lambda}_{N} &= \{ \bm{z} \in \Lambda_{N} \mid z_i(t) \neq z_j(t) \ \text{for} \ i \neq j \}, \ \text{and}\\
\hat{\Lambda}^{G}_{N} &= \Lambda_{N}^{G} \cap \hat{\Lambda}_{N}.
\end{align*}
The set $\hat\Lambda_{N}$ implies that each element has no collision.
As a consequence, a critical point of $\mathcal{A}_{[0,N]}$ in $\hat{\Lambda}_{N}^{G}$ is also a critical point in $\hat{\Lambda}_{N}$ if the $N$ masses are equal.
This fact follows from the Palais principle:
\begin{proposition}[Palais principle, \cite{Palais79}]
Let \( M \) be a Hilbert space with an inner product \( \langle \cdot, \cdot \rangle \),  
and let \( G \) be a group such that each \( g \in G \) is a linear operator on \( M \) satisfying  
\[
\langle gx, gy \rangle = \langle x, y \rangle \quad \text{for any } x, y \in M.
\]
Define
\[\Sigma = \{ x \in M \mid gx = x \ \ \text{for} \ g \in G \}.\]  
Suppose that \( f \colon M \to \mathbb{R} \) is of class \( C^1 \), \( G \)-invariant,  
and \( \Sigma \) is a closed subspace of \( M \).  
If \( p \in \Sigma \) is a critical point of \( f|_{\Sigma} \), 
then \( p \in \Sigma \) is also a critical point of \( f \).
\end{proposition}

\begin{example}
\label{ex:standard_group}
    Let \( G \) be the cyclic group, i.e., \( G = \langle g \mid g^N = 1 \rangle (=: \mathbb{Z}_{N}) \), and its actions are given by: 
\begin{align}
\label{eq:gn1}
   \tau(g)t = t-1, \quad \rho(g) = \mathrm{id}, \quad \text{and} \quad \sigma(g) = (0, 1, \dots, N-1). 
\end{align} 
Then, for any \( \bm{z} \in \Lambda_{N}^{\mathbb{Z}_{N}} \),  
\begin{align}
\label{eq:simple_choreo}
z_{j}(t) = z_0(t+j) \ \ \text{for} \ t \in \mathbb{R} \ \text{and} \ j \in \{0,1,\dots,N-1\},
\end{align}
and this implies that a critical point of \( \mathcal{A}_{[0,N]} \) in \( \Lambda_{N}^{\mathbb{Z}_{N}} \) describes a simple choreography if it has no collisions.  

Example \ref{ex:standard_group} is a standard setting in proofs of the existence of periodic solutions using the Palais principle.  
However, it is also known that the global minimizer in \( \Lambda_{N}^{\mathbb{Z}_{N}} \) is only the rotating regular \( N \)-gon.  
Thus, additional constraints are needed to obtain nontrivial periodic orbits through minimizing methods.
\end{example}

\begin{example}[Setting in \cite{Yu17}]
Set  
\[
D_N = \langle g, h \mid g^N = h^2 = 1, (gh)^2 = 1 \rangle,
\]  
where the actions of \( g \) are the same as in \eqref{eq:gn1} in Example \ref{ex:standard_group}, and  
\begin{align*}
    &\tau(h)t = -t + 1, \quad \rho(h)\bm{z} = \bar{\bm{z}}, \quad \text{and} \\
    &\sigma(h) = (0, N-1)(1, N-2) \cdots (\mathbf{n}, N-1-\mathbf{n}),
\end{align*}  
where \( \mathbf{n} = \lfloor (N-1)/2 \rfloor \).
Thus, any $\bm{z} \in \Lambda_{N}^{D_{N}}$ satisfies the following three properties.
Firstly, the actions of $h$ imply:
\begin{align}
\label{eq:simple_choreo_1}
    z_{j}(t)=\bar{z}_{N-1-j}(1-t) \ \text{for} \ t \in \mathbb{R} \ \text{and} \ j  \in \{0,1,\dots,N-1\}
\end{align}
and combining \eqref{eq:simple_choreo} and \eqref{eq:simple_choreo_1} yields:
\begin{align*}
    z_{j}(t) =\bar{z}_{N-j}(-t),
\end{align*}
especially,
\begin{align}
\label{eq:simple_choreo_2}
\mathrm{Re}(\dot{z}_{j}(0)) = - \mathrm{Re}(\dot{z}_{N-j}(0)) .
\end{align}
Secondly, the actions of $gh$ indicate:
\begin{align}
\label{eq:origin0}
    z_0(t)= \bar{z}_0(-t)  \ \text{for} \ t \in \mathbb{R}.
\end{align}
By $\eqref{eq:origin0}$, we get $\mathrm{Im}(z_0(0))=0$.

Clrarly, \( \Lambda_{N}^{D_{N}} \subset \Lambda_{N}^{\mathbb{Z}_{N}} \).  
Since the set \( \Lambda_{N}^{D_{N}} \)  contains an element representing the rotating \( N \)-gon,  
which is the global minimizer in \( \Lambda_{N}^{\mathbb{Z}_{N}} \),  
we  need additional assumptions.
\end{example}
\begin{definition}[The $\bm{\omega}$-topological constraints, \cite{Yu17}]
\label{def_constraints}
For any $\bm{\omega} \in \Omega_N$, $\bm{z} \in \Lambda_N^{D_N}$ is said to satisfy the {\textit{$\bm{\omega}$-topological constraints}} if 
\[
\mathrm{Im}(z_0(j/2)) = \omega_j |\mathrm{Im}(z_0(j/2))|
\ \text{for} \ j \in \{1,\dots,N-1\}.
\]
\end{definition}
The periodic orbits obtained in Theorem \ref{theorem:Yu} satisfy the $\bm{\omega}$-topological constraints and following monotonicity.
\begin{theorem}[\cite{Yu17}]
\label{thm:yu_submain}
For each ${\bm{\omega}} \in \Omega_{N}$, there exists at least one simple choreography 
${\bm{z}}=(z_j)_{j=0}^{N-1} \in \hat{\Lambda}_{N}^{D_{N}}$ 
satisfying $\eqref{eq:nbp}$, the $\bm{\omega}$-topological constraints 
and the following properties:
\begin{enumerate}
    \item $\mathrm{Re}(\dot{z}_0(t))>0 \ \text{for} \ t \in (0,N/2)$, and
    \item $\mathrm{Re}(\dot{z}_0(0))=\mathrm{Re}(\dot{z}_0(N/2))=0$.
\end{enumerate}
\end{theorem}

\begin{remark}
Since the periodic orbits in Theorem \ref{theorem_main-Yu} are collision-free,
    $\eqref{eq:simple_choreo}$ and $\eqref{eq:simple_choreo_1}$ imply $\mathrm{Im}(z_0(j/2)) \neq 0$ for any $j \in \{1, \dots, N-1\}$.
\end{remark}

Let $\bm{z}_{\bm{\omega}} \in \Lambda_{N}^{D_N}$ satisfy the properties of Theorem \ref{thm:yu_submain} for $\bm{\omega}$.
If $\bm{z}_{\bm{\omega}} \in \hat\Lambda_{N}^{D_N}$,
then it draws a trajectory that transits between line segments parallel to the $x$- or $y$-axis.
The shape of the trajectories can be easily understood 
by showing examples as below.

For \( z_j \colon \mathbb{R} / N \mathbb{Z} \to \mathbb{C} \),  
\( \bm{z} \colon \mathbb{R} / N \mathbb{Z} \to \mathbb{C}^N \),  
and \( t_1, t_2 \) with \( 0 \leq t_1 < t_2 \leq N \), define  
\[
z_j([t_1,t_2]) = \{ z_j(t) \mid t_1 \leq t \leq t_2 \},  
\quad j \in \{0,1, \dots, N-1\},
\]
and  
\[
\bm{z}([t_1,t_2]) = \{ \bm{z}(t) \mid t_1 \leq t \leq t_2 \}.
\]
We think of $\bm{z}([t_1,t_2])$ as the  $N$ oriented curves 
in  the complex plane $\C$.

 Note that if $\bm{z}(t)$ is a solution of $\eqref{eq:nbp}$, so is $\bar{\bm{z}}(-t)$.
    Thus we get 
    \begin{align}
    \label{eq:conjugate_symmetry}
        z_0([0,N])=z_0([0,N/2]) \cup \bar{z}_0([0,N/2]).
    \end{align}
    Moreover, by the property $(2)$ of Theorem \ref{thm:yu_submain}, $z_0([0,N])$ forms a smooth closed curve.

\begin{remark}
\label{rem:number_closedcurve}
Set $\omega_0=\omega_N=0$ and 
\[
z_0((t_1,t_2)) = \{ z_0(t) \mid t_1 < t < t_2 \}.
\]
The proof of Proposition 3.1 in \cite{Yu17} implies that for $j \in \{0,\dots,N-1\}$, 
the trajectory $z_0((j/2,(j+1)/2))$ crosses the $x$-axis exactly once if $\omega_j \omega_{j+1} = -1$,  
and does not cross the $x$-axis otherwise. 
\end{remark}

\begin{example}
\label{ex:N=3}
    Set $N=3$ and $\bm{\omega}=(1,-1)$.
    Figure \ref{fig_orbit-1_-1}(1) represents $z_0([0,{N}/{2}])$.
    Each arrow precisely indicates $z_0([{i}/{2},{(i+1)}/{2}])$ 
    for $i=0,1,2$ and $\bm{\omega} =(1,-1) $ shows whether the trajectory passes through the positive or negative side of the $y$-axis.
    By $\eqref{eq:conjugate_symmetry}$, the case $\bm{\omega}=(1,-1)$ gives a periodic orbit like the figure-eight \cite{ChenMont00}.
    On the other hand, Figure \ref{fig_orbit-1_-1}(2) describes 
    the trajectory $\bm{z}_{\bm{\omega}}([0,1/2])$ 
    and each arrow represents
$z_0([0,1/2])$, $z_1([0,1/2])$ and $z_2([0,1/2])$, respectively. 
    The indices of the particles are determined from $\eqref{eq:simple_choreo}$. 
    See also Figure \ref{fig_orbit-1_-1}(3) for 
    $\bm{z}_{\bm{\omega}}([1/2,1])$.

\begin{figure}[htbp]
\begin{center}
\includegraphics[width=4.8in]{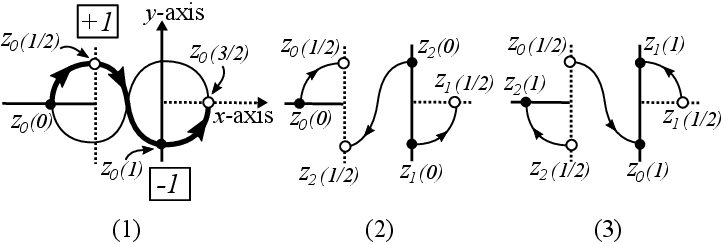}
\caption{Case $N=3$, $\bm{\omega}=(1,-1)$. 
(1) Thick arrows indiate $z_0([0,{N}/{2}])$. 
(2) $\bm{z}_{\bm{\omega}}([0,1/2])$. 
(3) $\bm{z}_{\bm{\omega}}([1/2,1])$.}
\label{fig_orbit-1_-1}
\end{center}
\end{figure}
\end{example}

\begin{example}
\label{ex:N=4}
Set \( N=4 \) and \( \bm{\omega}=(1,-1,1) \).  
As in the previous example, Figure \ref{fig_orbit-1_-1-1}(1) represents \( z_0([0,{N}/{2}]) \). 
Figures \ref{fig_orbit-1_-1-1}(2) and \ref{fig_orbit-1_-1-1}(3)  
depict \( \bm{z}_{\bm{\omega}}([0,1/2]) \) and \( \bm{z}_{\bm{\omega}}([1/2, 1]) \), respectively.  
While the figure-eight consists of two connected loops,  
a trajectory forming three such loops is called the super-eight,  
and the existence of a periodic solution with this shape in the  four-body problem has been established in  \cite{KZ03, Shibayama14}.

\begin{figure}[htbp]
\begin{center}
\includegraphics[width=4.8in]{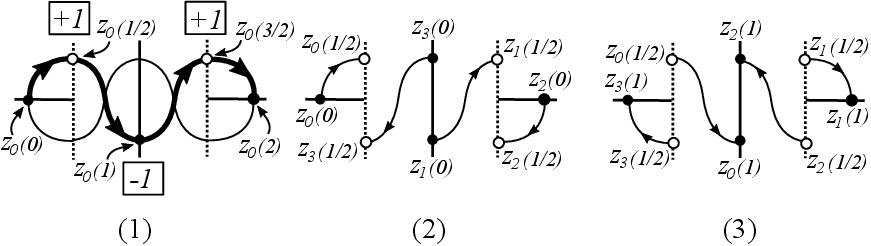}
\caption{Case $N=4$, $\bm{\omega}=(1,-1,1) $. 
 (1) Thick arrows indiate $z_0([0,{N}/{2}])$. 
(2) $\bm{z}_{\bm{\omega}}([0,1/2])$. 
(3) $\bm{z}_{\bm{\omega}}([1/2,1])$.}
\label{fig_orbit-1_-1-1}
\end{center}
\end{figure}
\end{example}

As seen above,  
each simple choreography in Theorem  \ref{theorem_main-Yu} travels a chain made of several loops. 
Moreover, Remark \ref{rem:number_closedcurve} implies that 
the number of loops is uniquely determined for each $\bm{\omega}$.  
For example, when $\bm{\omega} = (1, -1)$, the trajectory traces a chain made of two loops,  
whereas $\bm{\omega} = (1, -1, 1)$ results in a chain of three loops. 
More precisely,  
each trajectory of $\bm{z}_{\bm{\omega}}$ traces a chain consisting of $1 + |{\bm \omega}|$ loops,  
where $| {\bm \omega}|$ is defined by
\[
| {\bm \omega}| = \#\left\{ j \in \{1,\dots, N-2\} \,\middle|\, \omega_j \omega_{j+1} = -1 \right\}.
\]
In particular, if $\bm{\omega}= (1,1, \dots, 1)$, then 
$|\bm{\omega}|=0$ and 
$\bm{z}_{\bm{\omega}}$ traces a circle.

As shown in Examples~\ref{ex:N=3} and \ref{ex:N=4}, 
if $N$ is odd, then 
the periodic orbits $\bm{z}_{\bm{\omega}}$ for $\bm{\omega} \in \Omega_N$ 
are drawn based on $1$-solid and $1$-dotted horizontal line, $(N-1)/2$-solid and $(N-1)/2$-dotted vertical lines 
(Figure~\ref{fig_orbit-1_-1}(1)).  
When $N$ is even, the orbits are drawn based on $2$-solid horizontal lines, $(N-2)/2$-solid, and $N/2$-dotted vertical lines 
(Figure~\ref{fig_orbit-1_-1-1}(1)).  
The trajectory $\bm{z}_{\bm{\omega}}([(i/2),(i+1)/2])$
consists of $N$ oriented curves and
jumps from solid to dotted from $t={i}/{2}$ to $t=({i+1})/{2}$ if $i$
 is even 
 (see Figures~\ref{fig_orbit-1_-1}(2) and \ref{fig_orbit-1_-1-1}(2))
 and from dotted to solid if $i$
 is odd  
 (see Figures~\ref{fig_orbit-1_-1}(3) and \ref{fig_orbit-1_-1-1}(3)).

\begin{remark}
\label{rem_figure-super-8}
We provide  remarks on the figure-eight and the super-eight.
\begin{enumerate}
    \item The figure-eight, a simple choreography for the $3$-body problem with equal masses, 
    was discovered by Moore \cite{Moore93} as a numerical solution and later its existence was proven mathematically by Chenciner and Montgomery \cite{ChenMont00} using variational methods in a function space with a certain symmetry.  

    \item The existence of the super-eight, another simple choreography for the $4$-body problem with equal masses, was established by several researchers.  
    Gerver first discovered it numerically,
    Kapela and Zgliczy\'nski \cite{KZ03} provided a computer-assisted proof for its existence,
    and the third author \cite{Shibayama14} later gave a variational proof. 
    \item The reason why we write `like the figure-eight' and `like the super-eight' in Examples \ref{ex:N=3} and \ref{ex:N=4} is
    that the obtained orbits in Theorem \ref{thm:yu_submain} only have symmetry with respect to the $x$-axis (see (\ref{eq:origin0})),
    whereas the figure-eight and super-eight have symmetry with respect to both the $x$- and $y$-axis.
    Moreover, it is not clear whether the vertical solid and dotted lines are evenly spaced.   
    However, by considering additional symmetries, we obtain the same periodic solutions as the figure-eight and super-eight.
  See  \cite[Section 3]{Yu17} for more details.
\end{enumerate}
\end{remark}

To clarify the number $2^{N-3} + 2^{\lfloor(N-3)/2\rfloor}$  in Theorem \ref{theorem:Yu}, 
we define elements $- \bm{\omega}, \widehat{\bm{\omega}} \in \Omega_N$ 
for each $\bm{\omega}= (\omega_1, \dots, \omega_{N-1}) \in \Omega_N$ 
as follows. 
\begin{align*}
    - \bm{\omega}&:=(-\omega_1,-\omega_2,\dots, -\omega_{N-1}),\  \mbox{and}\\
    \widehat{\bm{\omega}}&:=
    (\omega_{N-1},\omega_{N-2}, \dots,\omega_1).
\end{align*}

\begin{figure}[htbp]
  \begin{minipage}[b]{0.48\columnwidth}
  \centering
\includegraphics[scale=0.25]{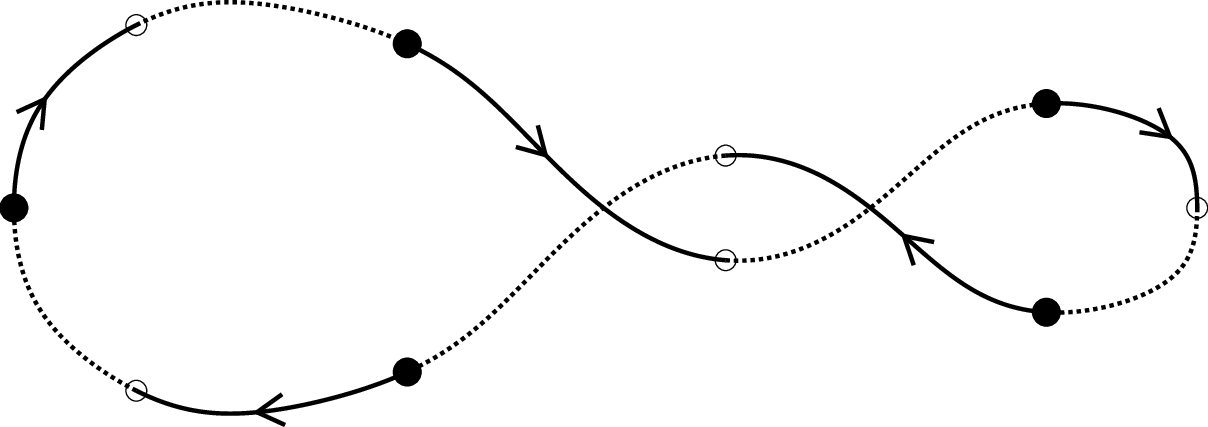} 

(1) $\bm{z}_{\bm{\omega}}(t)$. 
\end{minipage}
  \begin{minipage}[b]{0.48\columnwidth}
  \centering
\includegraphics[scale=0.25]{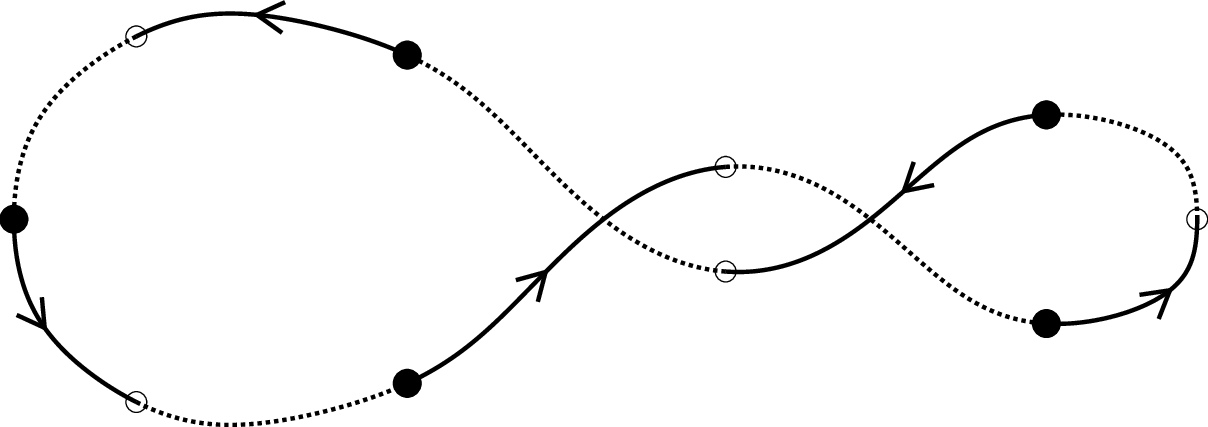} 

(2) $\bm{z}_{-\bm{\omega}}(t)$. 
\end{minipage}
\vspace{0.5cm}

  \begin{minipage}[b]{0.48\columnwidth}
  \centering
\includegraphics[scale=0.25]{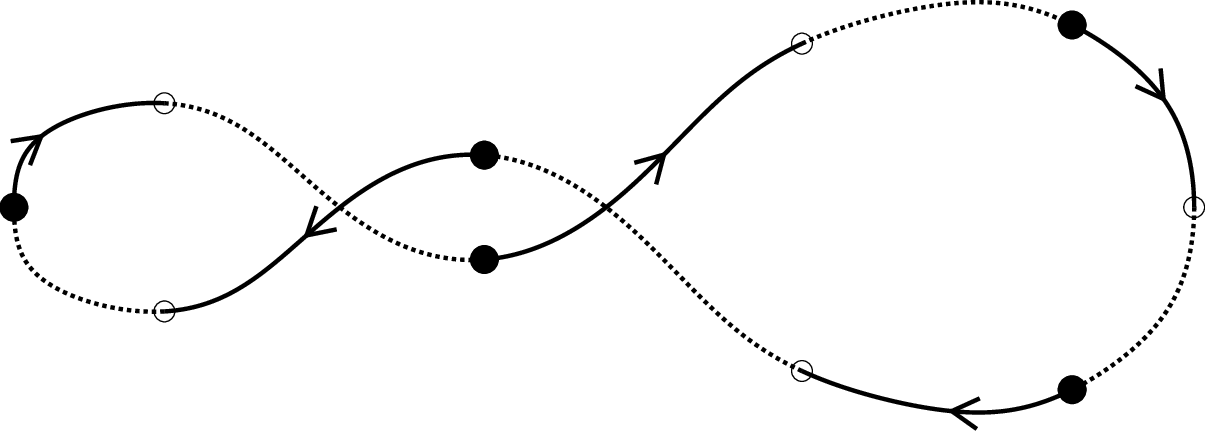}

(3) $\bm{z}_{\widehat{\bm{\omega}}}(t)$. 
\end{minipage}

\caption{Case $N=5$, $\bm{\omega}=(1, 1, -1, 1)$. 
}
\label{fig_equivalent-solution}
\end{figure}

We say that elements $\bm{\omega}, \bm{\omega}' \in \Omega_N$ 
are \textit{equivalent} and write $\bm{\omega} \sim \bm{\omega}' $ 
if $\bm{\omega}' \in \{\pm \bm{\omega}, \pm \widehat{\bm{\omega}}\}$. 
Figure~\ref{fig_equivalent-solution} shows examples of simple choreographies 
$\bm{z}_{\bm{\omega}}(t)$ corresponding to equivalent elements 
$\bm{\omega}$, $-\bm{\omega}$ and $\widehat{\bm{\omega}}$. 
The number of elements in \( \Omega_{N} \) up to the equivalence relation 
$\sim$ is given by  
\[
2^{N-3} + 2^{\lfloor (N-3)/2 \rfloor}.
\]  
Thus, this number represents the number of simple choreographies up to the equivalence relation $\sim$ obtained in Theorem \ref{thm:yu_submain}.  
See Figures~\ref{fig:numsoln=3}, \ref{fig:numsoln=4},
\ref{fig:numsoln=5} and \ref{fig:numsoln=6} 
in the case $N= 3, 4, 5$ and $6$, respectively.  
Moreover, Theorem \ref{theorem:Yu} follows immediately from Theorem \ref{thm:yu_submain}.

Figure~\ref{fig:numsoln=19} shows particular examples of simple choreographies 
$\bm{z}_{\bm{\omega}}(t)$ for 
$\bm{\omega}= \bm{\omega}_{\min}, \bm{\omega}_{\max}$ in the case $N=19$.

\begin{figure}[htbp]
\begin{center}
  \begin{minipage}[b]{0.48\columnwidth}
  \centering
\includegraphics[scale=0.2]{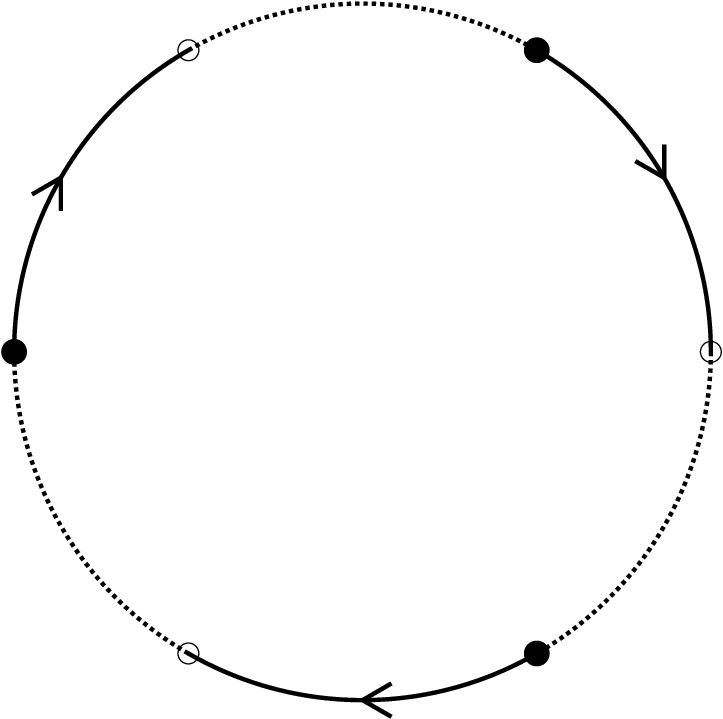} 

(1) $\bm{\omega}=(1, 1)$. 
\end{minipage}
  \begin{minipage}[b]{0.48\columnwidth}
  \centering
\includegraphics[scale=0.2]{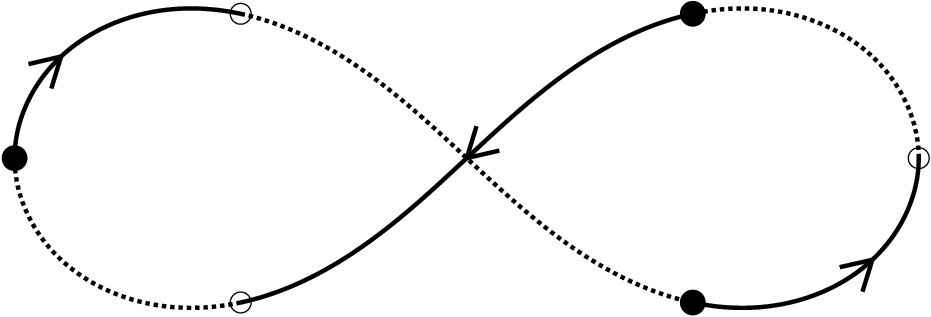}

(2) $\bm{\omega}=(1, -1)$.
\end{minipage}

\caption{
Simple choreographies $\bm{z}_{\bm{\omega}}(t)$  
up to the equivalence $ \sim$ in the case $N=3$. 
}
\label{fig:numsoln=3}
\end{center}
\end{figure}

\begin{figure}[htbp]
\begin{center}

  \begin{minipage}[b]{0.48\columnwidth}
  \centering
\includegraphics[scale=0.2]{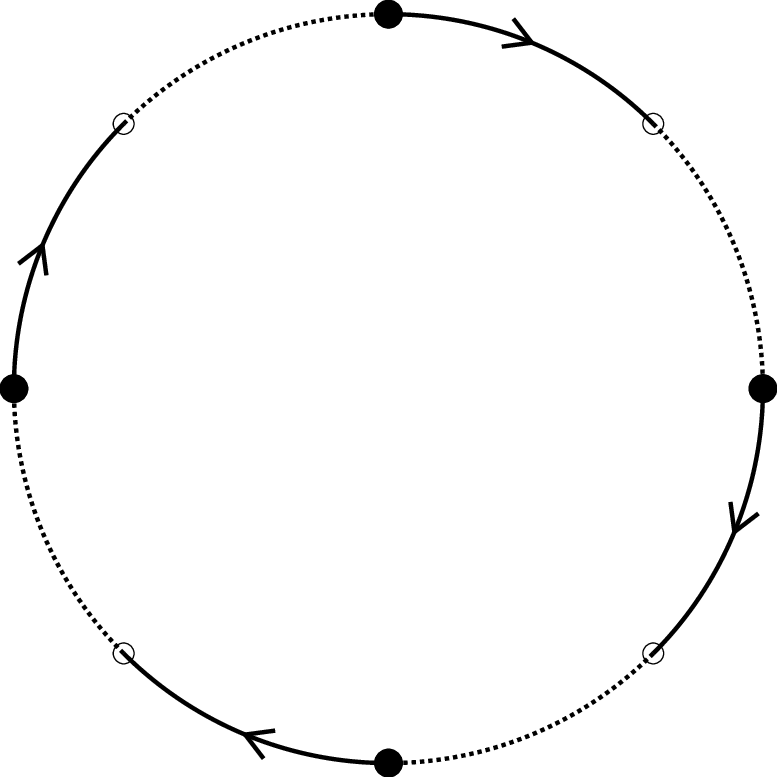} 

(1)  $\bm{\omega}=(1, 1, 1)$. 
\end{minipage}
  \begin{minipage}[b]{0.48\columnwidth}
  \centering
\includegraphics[scale=0.2]{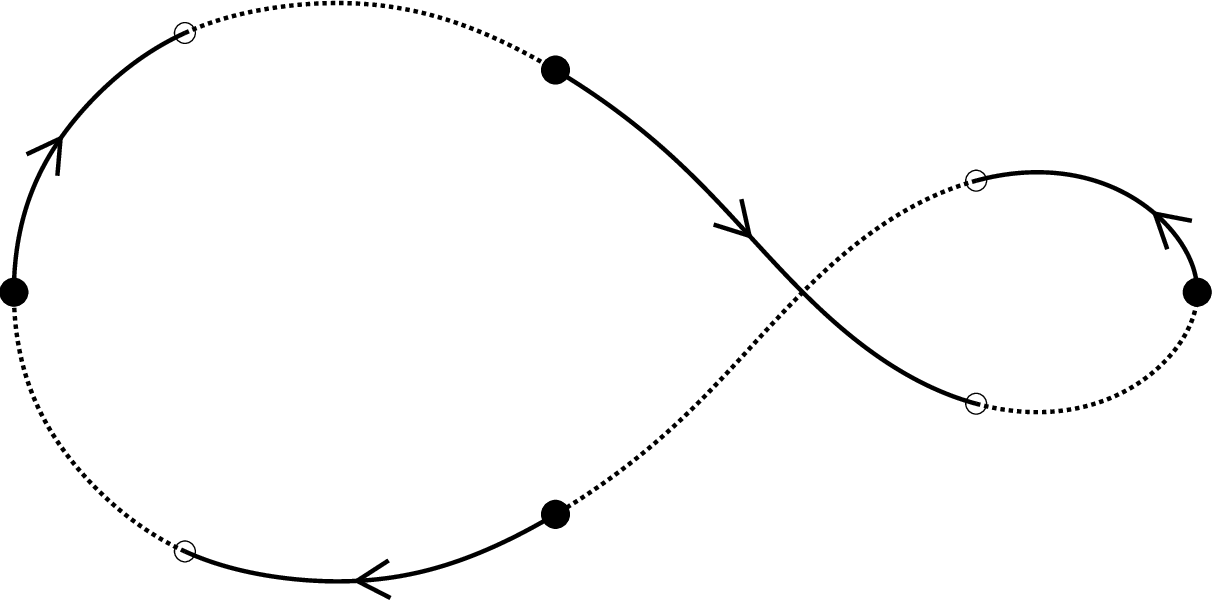} 

(2) $\bm{\omega}=(1, 1, -1)$. 
  \end{minipage}
\vspace{0.5cm}
  
  \begin{minipage}[b]{0.48\columnwidth}
  \centering
\includegraphics[scale=0.2]{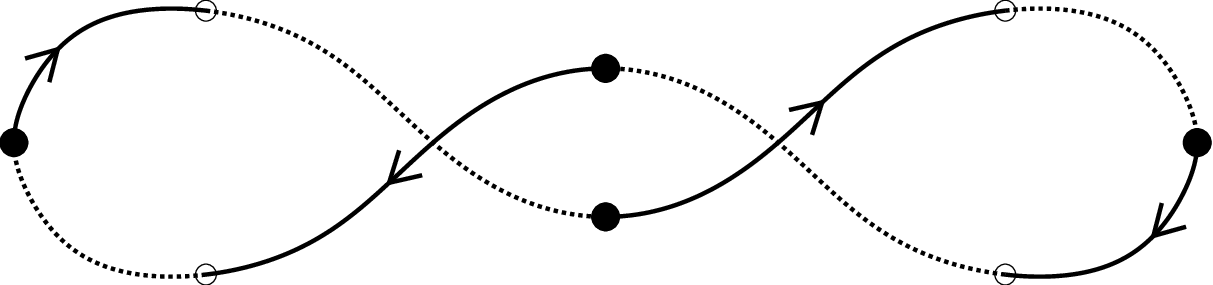} 

(3) $\bm{\omega}=(1, -1, 1)$.
\end{minipage}

\caption{
Simple choreographies $\bm{z}_{\bm{\omega}}(t)$ 
up to the equivalence $ \sim$ in the case $N=4$ . 
}

\label{fig:numsoln=4}
\end{center}
\end{figure}

\begin{figure}[htbp]
\begin{center}

  \begin{minipage}[b]{0.48\columnwidth}
  \centering
\includegraphics[scale=0.2]{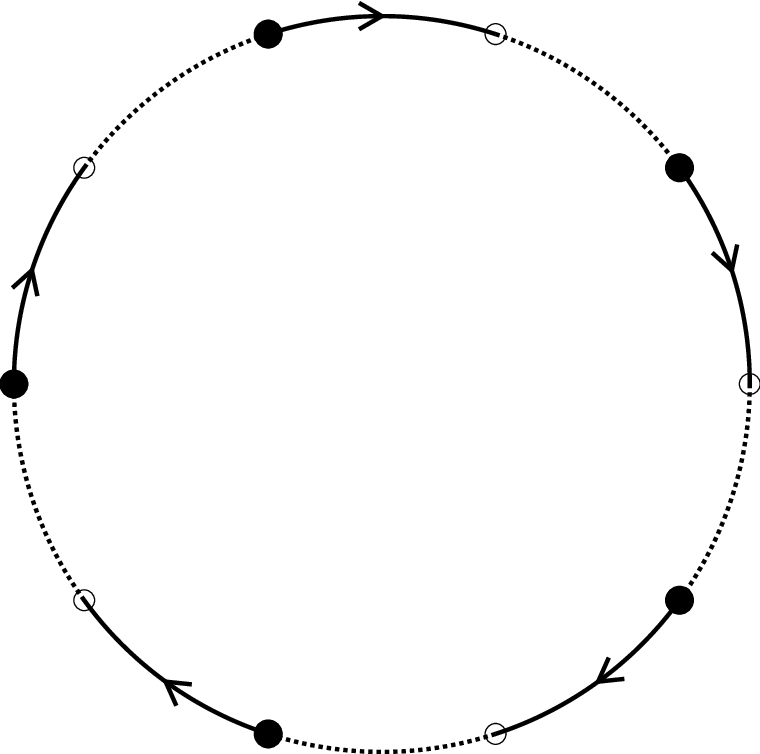} 

(1) $\bm{\omega}=(1, 1, 1,1)$. 
\end{minipage}
  \begin{minipage}[b]{0.48\columnwidth}
  \centering
\includegraphics[scale=0.2]{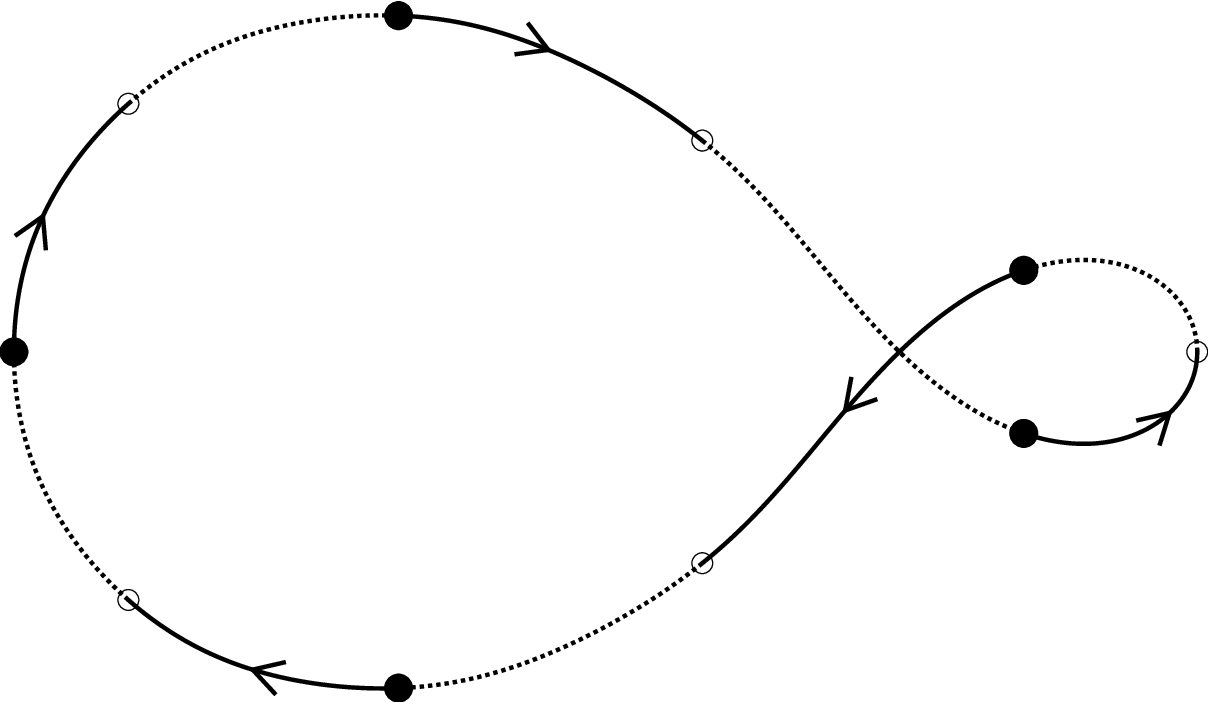} 

(2) $\bm{\omega}=(1, 1,1,  -1)$. 
\end{minipage}
\vspace{0.5cm}

  \begin{minipage}[b]{0.48\columnwidth}
  \centering
\includegraphics[scale=0.2]{n=5ppmp.eps} 

(3) $\bm{\omega}=(1, 1, -1, 1)$. 
\end{minipage}
\vspace{0.5cm}

  \begin{minipage}[b]{0.48\columnwidth}
  \centering
\includegraphics[scale=0.2]{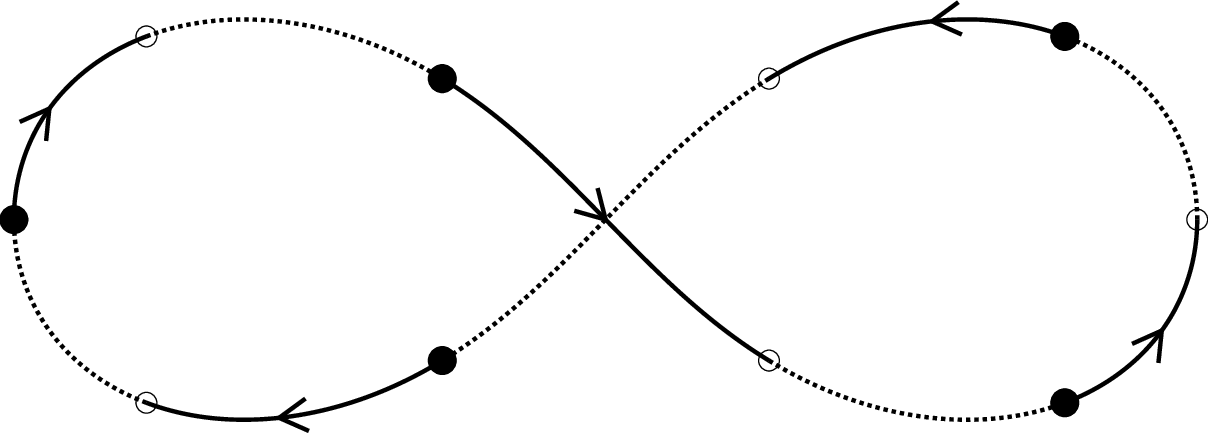} 

(4) $\bm{\omega}=(1, 1, -1, -1)$.
\end{minipage}
  \begin{minipage}[b]{0.48\columnwidth}
  \centering
\includegraphics[scale=0.2]{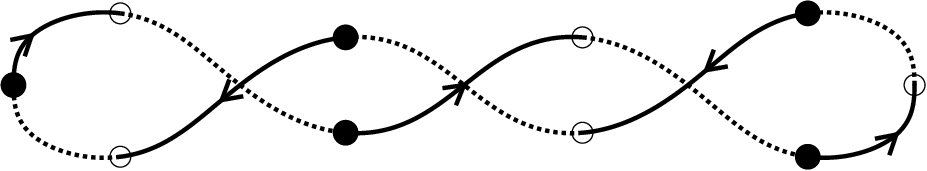} 

(5) $\bm{\omega}=(1, -1, 1, -1)$.
\end{minipage}
\vspace{0.5cm}

  \begin{minipage}[b]{0.48\columnwidth}
  \centering
\includegraphics[scale=0.2]{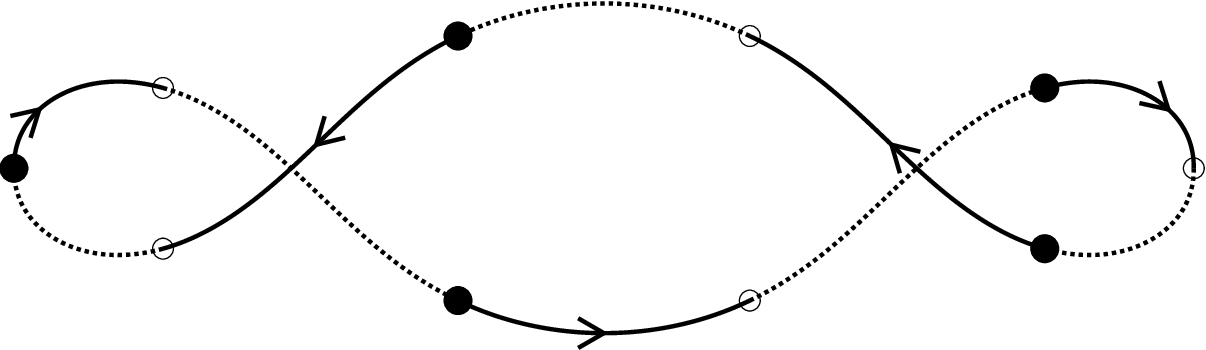} 

(6) $\bm{\omega}=(1, -1, -1, 1)$.
\end{minipage}

\caption{
Simple choreographies $\bm{z}_{\bm{\omega}}(t)$ 
up to the equivalence $ \sim$ in the case $N=5$.
}

\label{fig:numsoln=5}
\end{center}
\end{figure}

\begin{figure}[htbp]
\begin{center}

  \begin{minipage}[b]{0.48\columnwidth}
    \centering
\includegraphics[scale=0.2]{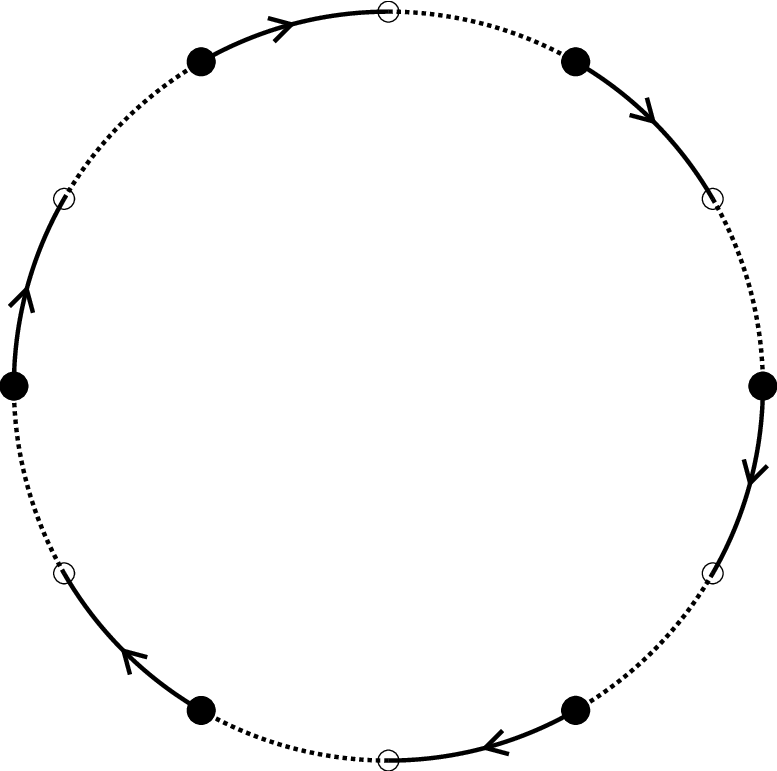} 

(1) $\bm{\omega}=(1, 1, 1,1, 1)$.  
\end{minipage}
  \begin{minipage}[b]{0.48\columnwidth}
  \centering
\includegraphics[scale=0.2]{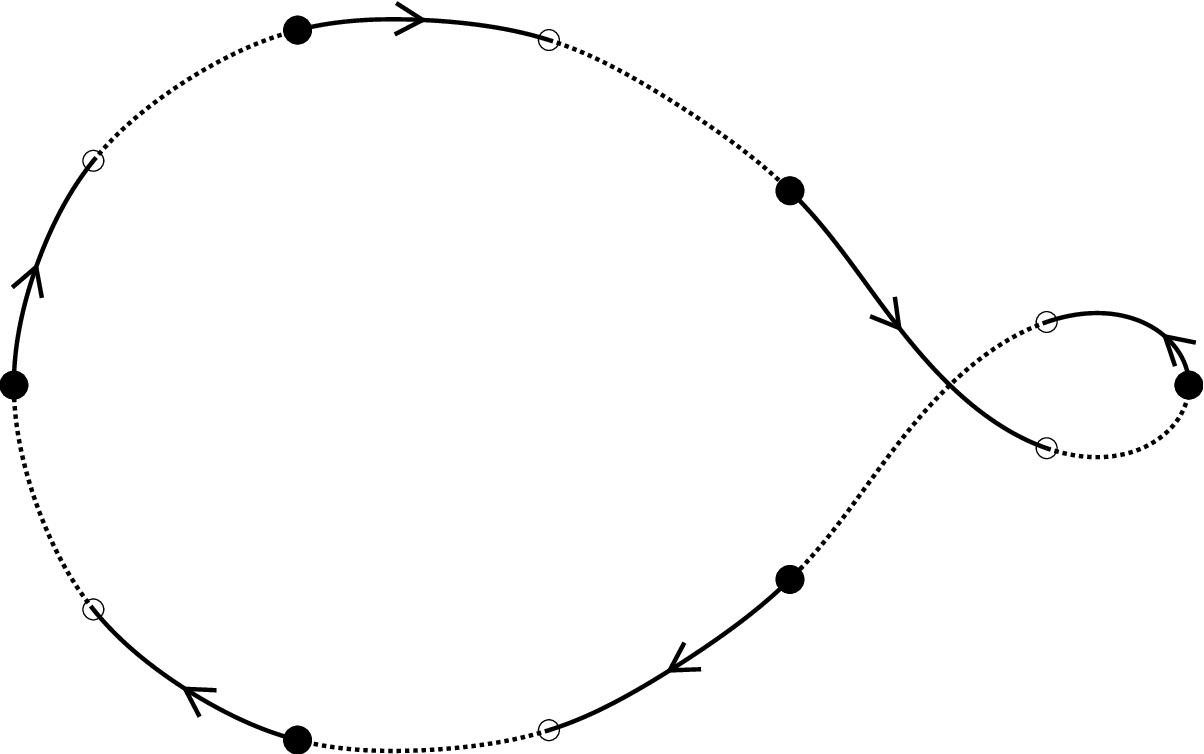} 

(2) $\bm{\omega}=(1, 1, 1,1, -1)$.
\end{minipage}
\vspace{0.5cm}

  \begin{minipage}[b]{0.48\columnwidth}
  \centering
\includegraphics[scale=0.2]{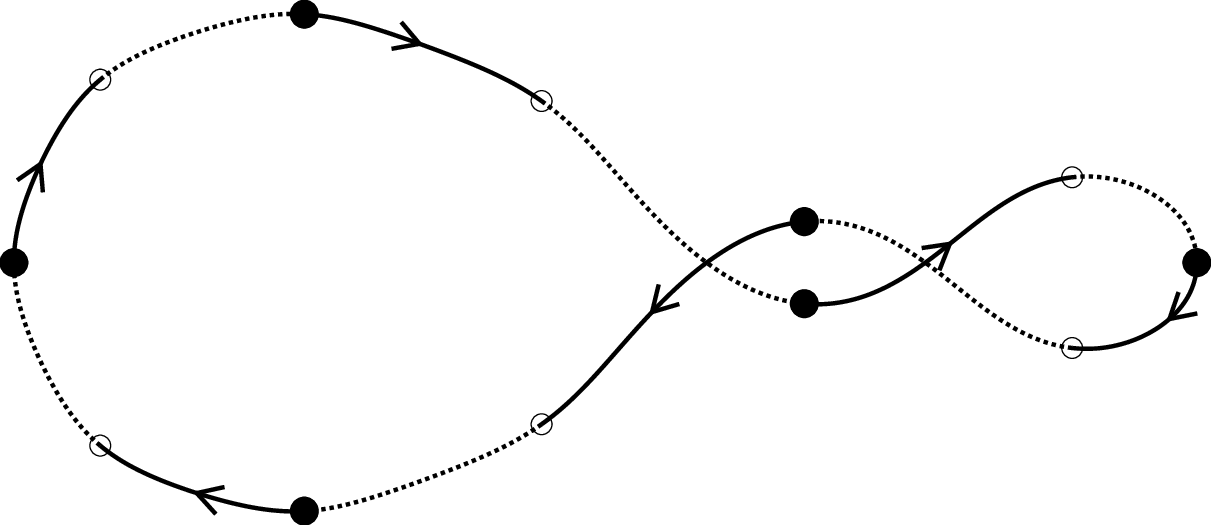} 

(3) $\bm{\omega}=(1, 1, 1,-1, 1)$.  
\end{minipage}
  \begin{minipage}[b]{0.48\columnwidth}
  \centering
\includegraphics[scale=0.2]{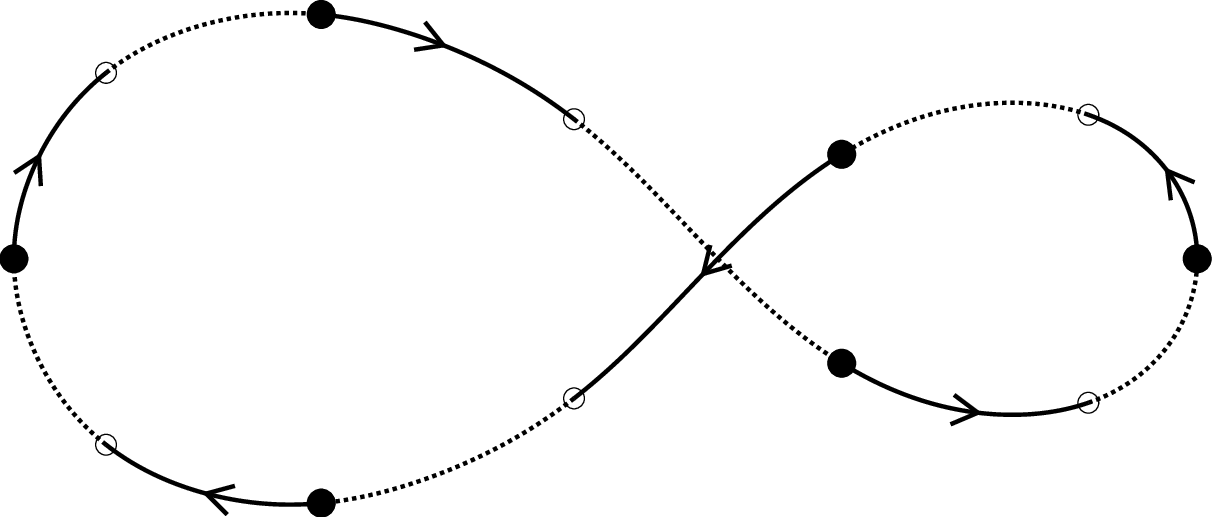} 

(4) $\bm{\omega}=(1, 1, 1,-1, -1)$.
\end{minipage}
\vspace{0.5cm}

  \begin{minipage}[b]{0.48\columnwidth}
  \centering
\includegraphics[scale=0.2]{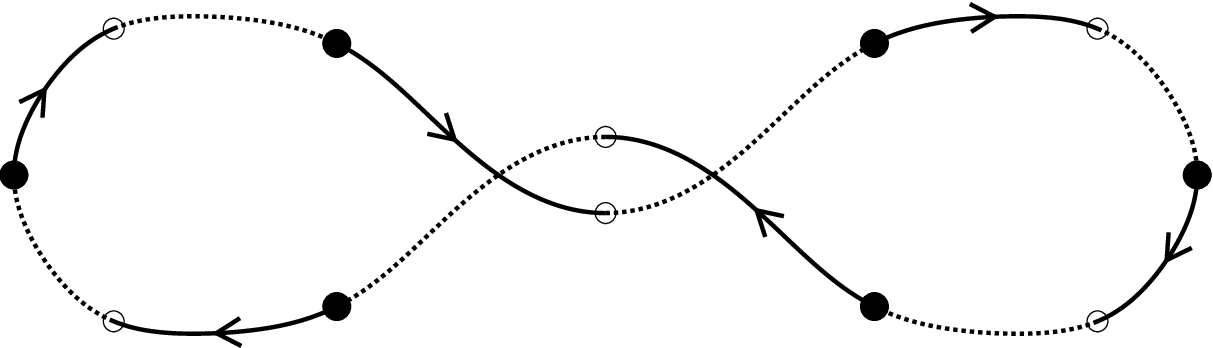} 

(5) $\bm{\omega}=(1, 1,-1,  1,1)$. 
\end{minipage}
  \begin{minipage}[b]{0.48\columnwidth}
  \centering
\includegraphics[scale=0.2]{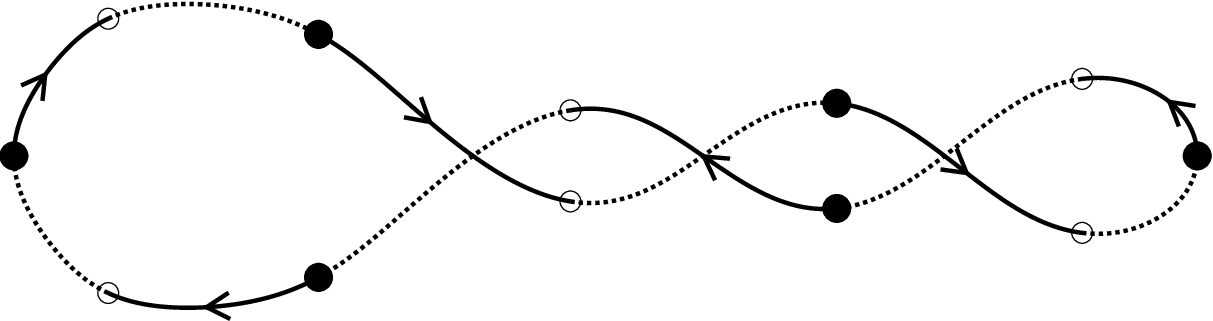} 

(6) $\bm{\omega}=(1, 1,-1,1, -1)$.
\end{minipage}
\vspace{0.5cm}

  \begin{minipage}[b]{0.48\columnwidth}
  \centering
\includegraphics[scale=0.2]{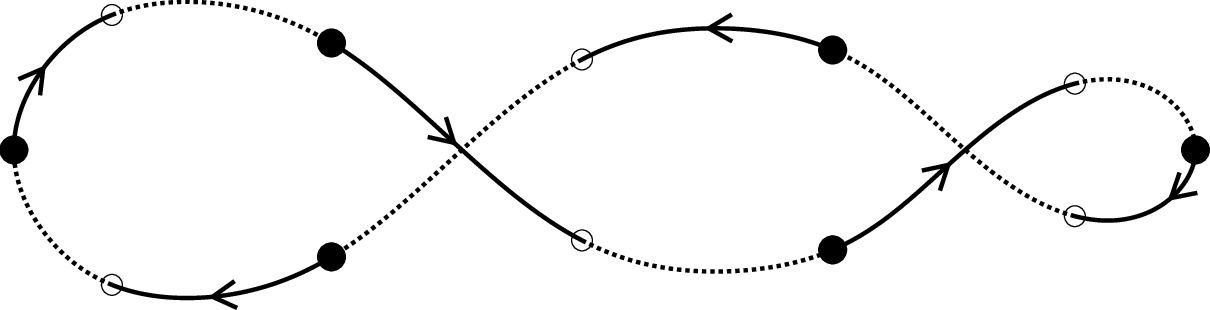} 

(7) $\bm{\omega}=(1, 1, -1,- 1, 1)$. 
\end{minipage}
  \begin{minipage}[b]{0.48\columnwidth}
  \centering
\includegraphics[scale=0.2]{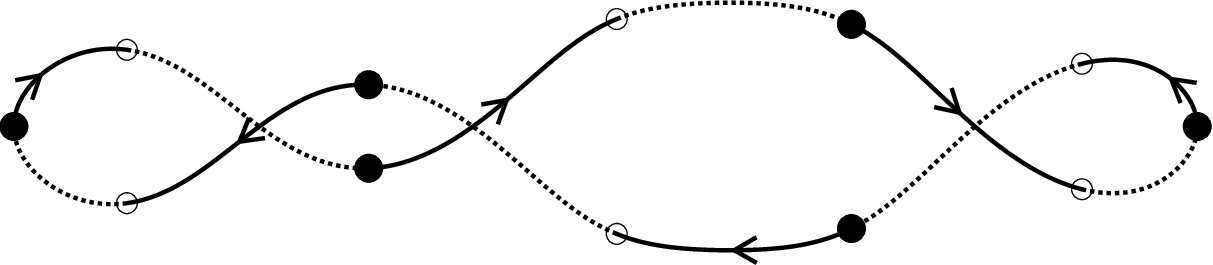} 

(8) $\bm{\omega}=(1, -1, 1,1, -1)$.
\end{minipage}
\vspace{0.5cm}

  \begin{minipage}[b]{0.48\columnwidth}
  \centering
\includegraphics[scale=0.2]{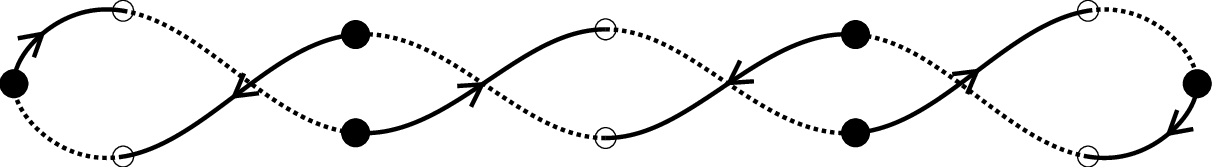} 

(9) $\bm{\omega}=(1, -1, 1,-1, 1)$.
\end{minipage}
  \begin{minipage}[b]{0.48\columnwidth}
  \centering
\includegraphics[scale=0.2]{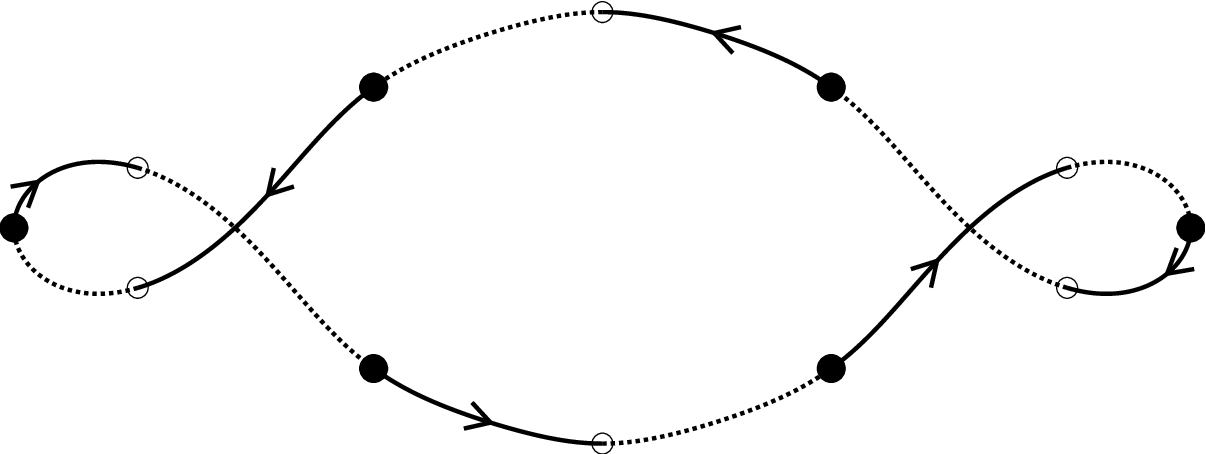} 

(10) $\bm{\omega}=(1, -1, -1,-1, 1)$.  
\end{minipage}

\caption{
Simple choreographies $\bm{z}_{\bm{\omega}}(t)$ 
up to the equivalence $ \sim$ in the case $N=6$.  
      }

\label{fig:numsoln=6}

\end{center}
\end{figure}

\begin{lemma}
\label{lem_stretchfactor-equivalent}
Suppose that $\bm{\omega}$ and   $\bm{\omega}'$ are equivalent. 
If the braid $\alpha_{\bm{\omega}}$ is pseudo-Anosov, 
then $\alpha_{\bm{\omega}'}$ is also pseudo-Anosov. 
Moreover, $\alpha_{\bm{\omega}}$ and $\alpha_{\bm{\omega}'}$ have 
the same stretch factor.
\end{lemma}

\begin{proof}
Let $\bm{\omega}= (\omega_1, \dots, \omega_{N-1})$. 
Suppose that $\bm{\omega}'= -\bm{\omega}$. 
The braid $\alpha_{-\bm{\omega}}$ is the mirror of $\alpha_{\bm{\omega}}$, 
i.e., 
$\alpha_{-\bm{\omega}}$ is obtained from $\alpha_{\bm{\omega}}$ 
by changing the sign of each crossing in $\alpha_{\bm{\omega}}$. 
Then the assertion holds, since the pseudo-Anosov property 
and the stretch factor are preserved under the mirror. 

Suppose that $\bm{\omega}'= -\widehat{\bm{\omega}} 
= (-\omega_{N-1}, -\omega_{N-2}, \dots, - \omega_1)$. 
Then 
$$\alpha_{- \widehat{\bm{\omega}}} = \sigma_1^{-\omega_{N-1}} 
\sigma_2^{-\omega_{N-2}} \cdots \sigma_{N-1}^{-\omega_1}.$$ 
On the other hand, 
the inverse of $\alpha_{\bm{\omega}}= 
\sigma_1^{\omega_1} \sigma_2^{\omega_2} \dots \sigma_{N-1}^{\omega_{N-1}}$ is given by 
$$\alpha_{\bm{\omega}}^{-1}= \sigma_{N-1}^{-\omega_{N-1}} 
\sigma_{N-2}^{-\omega_{N-2}} \cdots \sigma_1^{- \omega_1}.$$
Then 
$\Delta \alpha_{\bm{\omega}}^{-1} \Delta^{-1} 
= \alpha_{- \widehat{\bm{\omega}}}$, 
that is 
$\alpha_{-\widehat{\bm{\omega}}}$ is conjugate to 
$\alpha_{\bm{\omega}}^{-1}$. 
Clearly, if $\alpha_{\bm{\omega}}$ is a pseudo-Anosov braid, 
then $\alpha_{\bm{\omega}}^{-1}$ is also a pseudo-Anosov braid 
with the same stretch factor as that of $\alpha_{\bm{\omega}}$. 
Hence, the assertion follows 
since $\alpha_{-\widehat{\bm{\omega}}}$ is conjugate to $\alpha_{\bm{\omega}}^{-1}$.

Finally, we suppose that $\bm{\omega}'= \widehat{\bm{\omega}}$. 
Since $\alpha_{- \widehat{\bm{\omega}}}$ is the mirror of 
$\alpha_{\widehat{\bm{\omega}}}$, 
the assertion follows from the above two cases. 
This completes the proof. 
\end{proof}

\section{Proofs}
\label{section_proofs}

Let $Z_{\bm{\omega}}$ be the braid type of the 
 simple choreography $\bm{z}_{\bm{\omega}}(t)$ 
 for each $\bm{\omega} \in \Omega_N$. 
For the proof of Theorem~\ref{thm_main}, 
we first prove the following result which tells us a 
 representative of $Z_{\bm{\omega}}$.

\begin{theorem}
\label{thm_braid-type-omega}
For each $\bm{\omega}= (\omega_1, \dots, \omega_{N-1}) \in \Omega_N$, 
the primitive braid type of the simple choreography $\bm{z}_{\bm{\omega}}(t)$ 
is given by 
$\alpha_{- \bm{\omega}}= \sigma_1^{-\omega_1} \sigma_2^{-\omega_2} 
 \cdots\sigma_{N-1}^{-\omega_{N-1}}$. 
In particular, the braid type $Z_{\bm{\omega}}$ of $\bm{z}_{\bm{\omega}}(t)$  is represented by 
$(\alpha_{- \bm{\omega}})^N$. 
\end{theorem}

\begin{figure}[htbp]
\begin{center}
\includegraphics[width=4in]{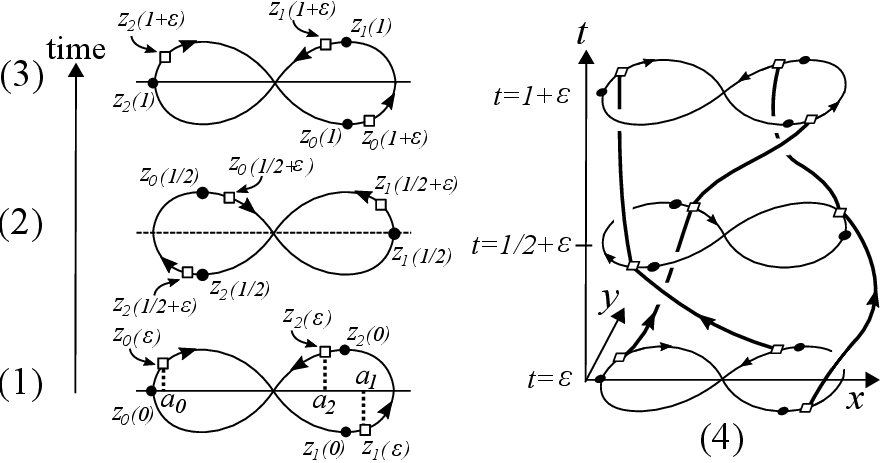}
\caption{
Case $N=3$, $\bm{\omega}= (1, -1)$. 
(1) $\bm{z}_{\bm{\omega}}(\epsilon)$. 
($a_i$ is the projection of $z_i(\epsilon)$ for $i= 0,1,2$.)
(2) 
$\bm{z}_{\bm{\omega}}(\frac{1}{2}+\epsilon)$. 
(3) 
$\bm{z}_{\bm{\omega}}(1+ \epsilon)=\bm{z}_{\bm{\omega}}(\epsilon)$. 
(4) Braid $b_{[\epsilon, 1+ \epsilon]} 
= b(\bm{z}_{\bm{\omega}}([\epsilon, \epsilon+1]))$.}
\label{fig_1_-1}
\end{center}
\end{figure}

\begin{figure}[htbp]
\begin{center}
\includegraphics[width=4in]{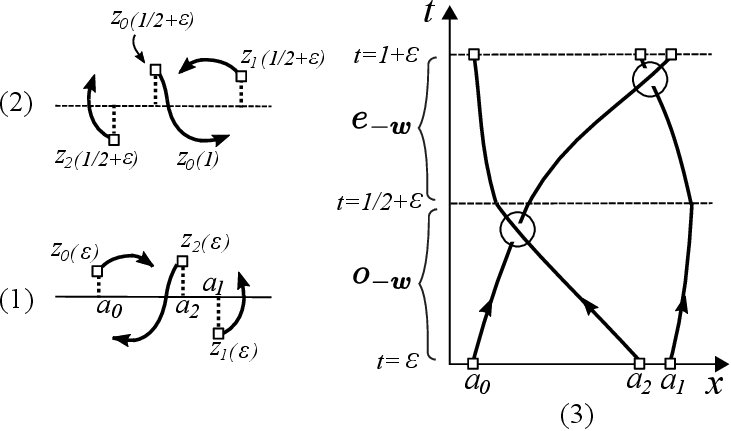}
\caption{Case $N=3$, $\bm{\omega}= (1, -1)$. 
(1) $\bm{z}_{\bm{\omega}}([\epsilon, \frac{1}{2}+ \epsilon])$. 
(2) $\bm{z}_{\bm{\omega}}([\frac{1}{2}+\epsilon, 1+ \epsilon])$. 
(3) Projection $\overline{b_{[\epsilon, 1+ \epsilon]}}$  on the $xt$-plane. 
$\overline{b_{[\epsilon, 1+ \epsilon]}}= e_{- \bm{\omega}} \cdot 
o_{- \bm{\omega}}= \sigma_2 \cdot \sigma_1^{-1}$ in this case. 
Small circles indicate the double points.}
\label{fig_construction}
\end{center}
\end{figure}

\begin{proof}
The simple choreography $\bm{z}_{\bm{\omega}}(t)$ 
associated with $\bm{\omega} \in \Omega_N$ 
has the period $N$ and  the primitive period $1$.  
Choosing a small $\epsilon >0$, 
we  consider a primitive braid 
$b(\bm{z}_{\bm{\omega}}([\epsilon, \epsilon+1]))$ of $\bm{z}_{\bm{\omega}}(t)$ (see Section~\ref{subsection_particle-dances}).
Figure~\ref{fig_1_-1} illustrates 
$b(\bm{z}_{\bm{\omega}}([\epsilon, \epsilon+1]))$ 
in the case $\bm{\omega}=(1, -1)$. 
For simplicity, we write 
$$b_{[ \epsilon, 1+ \epsilon]}:= b(\bm{z}_{\bm{\omega}}([\epsilon, \epsilon+1])).$$
We now prove that 
the braid type $\langle b_{[ \epsilon, 1+ \epsilon]} \rangle$ 
is given by $\alpha_{- \bm{\omega}}$
(see (\ref{equation_alpha}) for the braid $\alpha_{- \bm{\omega}}$). 
By (\ref{eq:simple_choreo}) and (\ref{eq:simple_choreo_1}), 
it holds 
$$z_{j}(0)=\bar{z}_{N-1-j}(1)= \bar{z}_{N-j}(0)\  \mbox{for}\ 
j \in \{0, 1, \dots, N-1\}.$$
Hence  we have 
$$\mathrm{Re}(z_{j}(0))= \mathrm{Re}({z}_{N-j}(0)) \ \mbox{for}\ 
j \in \{0, 1, \dots, N-1\}.
$$
If $\epsilon>0$ is small, then 
$\mathrm{Re}(z_{j}(\epsilon)) \ne  \mathrm{Re}({z}_{N-j}(\epsilon))$ 
by \eqref{eq:simple_choreo_2}. 
Moreover 
by the properties of the simple choreography $\bm{z}_{\bm{\omega}}(t)$ 
(see Section~\ref{section_Yu}), 
if $i,j \in \{0, \dots, N-1\}$ with $i \ne j$, then 
$\mathrm{Re}(z_{i}(\epsilon)) \ne  \mathrm{Re}({z}_{j}(\epsilon))$.

Let us set $\mathrm{Re}(z_i(\epsilon)) = p_i$ and 
let $a_i= (p_i, 0) \in {\Bbb R} \times \{0\}$ be the projection of $z_i(\epsilon) \in {\Bbb C} \simeq {\Bbb R}^2$ on the first component. 
Then $ \{a_0, \dots, a_{N-1}\} \subset {\Bbb R} \times \{0\}$ 
is a set of $N$ points, which is denoted by $A_N$. 
 Consider the projection 
 of the braid $b_{[ \epsilon, 1+ \epsilon]}$ onto the $xt$-plane. 
 Note that the projection  contains at most double points, 
 since the periodic orbits is a simple choreography and 
 the closed curve $z_0([0, N])$ on which $N$ particles lie 
 contain at most double points. 
 At each double point of the projection, we indicate the over/under crossings 
 determined by the braid 
 $b_{[ \epsilon, 1+ \epsilon]}$, as shown in Figure~\ref{fig_construction}(3). 
Then the result is a braid (with base points $A_N$)  
denoted by $\overline{b_{[ \epsilon, 1+ \epsilon]}} 
=  \overline{b(\bm{z}_{\bm{\omega}}([\epsilon, \epsilon+1]))}$. 
We also call the braid $\overline{b_{[ \epsilon, 1+ \epsilon]}}$ 
the {\it projection} of  $b_{[ \epsilon, 1+ \epsilon]}$. 
Note that  $\overline{b_{[ \epsilon, 1+ \epsilon]}}$ has the same braid type as 
$b_{[ \epsilon, 1+ \epsilon]}$.

\medskip
\noindent
{\bf Claim.}
For each $\bm{\omega} \in \Omega_N$ 
the braids $\overline{b_{[ \epsilon, 1+ \epsilon]}} 
(= \overline{b(\bm{z}_{\bm{\omega}}([\epsilon, \epsilon+1]))})$ 
and   $\alpha_{\bm{-\omega}} $ are conjugate in $B_N$. 
\medskip

\noindent
Proof of Claim. 
We consider the trajectory $\bm{z}_{\bm{\omega}}([0, \frac{N}{2}])$. 
See Figure~\ref{fig_trajectory}. 
We focus on the motion of $0$th particle $z_0(t)$. 
Recall that  $z_0(t)$ satisfies the two conditions (1) and (2) in Theorem~\ref{thm:yu_submain}. 
Between $t=0$ and $t= \frac{N}{2}$, 
the $0$th particle passes by  all other $N-1$ particles. 
More precisely, for each $j=1, \dots, N-1$, 
when $z_0(t)$ and the $(N-j)$th particle $z_{N-j}(t)$ pass each other,  
they  lie on the same vertical line at $t= \frac{j}{2}$, 
i.e., $\mathrm{Re}(z_0(j/2))= \mathrm{Re} (z_{N-j}(j/2))$. 
The $\bm{\omega}$-topological constraints (Definition~\ref{def_constraints}) 
tell us the following inequality 
about the imaginary part: 
\begin{eqnarray*}
\mathrm{Im}(z_0(\tfrac{j}{2})) &>& 
\mathrm{Im}(z_{N-j}(\tfrac{j}{2})) \hspace{5mm} \mbox{if} \ \omega_j =1, 
\\
\mathrm{Im}(z_0(\tfrac{j}{2})) &<& 
\mathrm{Im}(z_{N-j}(\tfrac{j}{2})) \hspace{5mm} \mbox{if} \ \omega_j =-1.
\end{eqnarray*}
This means that 
$z_0(t)$ passes over $z_{N-j}(t)$ at $t= \tfrac{j}{2}$ 
if $\omega_j=1$. 
Otherwise, it passes under $z_{N-j}(t)$ at $t= \tfrac{j}{2}$. 
These together with 
the choreographic condition (\ref{eq:simple_choreo}) 
enables us to 
read off the braid word from $\overline{b_{[ \epsilon, 1+ \epsilon]}}$, 
as we will see now. 

\begin{figure}[htbp]
\begin{center}
\includegraphics[width=4.8in]{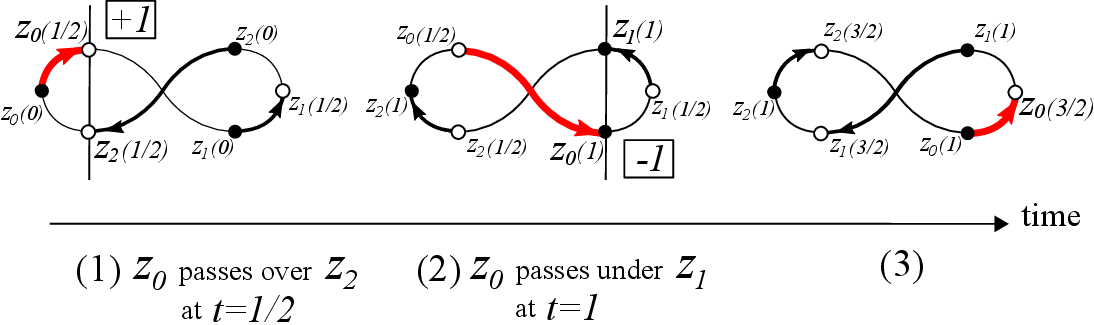}
\caption{
$\bm{z}_{\bm{\omega}}([0, \frac{N}{2}])$ when $N=3$, $\bm{\omega}= (1, -1)$. 
In this case, 
the $0$th particle 
$z_0(t)$ passes over $z_2(t)$ at $t= \frac{1}{2}$. 
Then it passes under $z_1(t)$ at $t= 1$. 
(1) $\bm{z}_{\bm{\omega}}([0, \frac{1}{2}])$. 
(2) $\bm{z}_{\bm{\omega}}([\frac{1}{2}, 1])$. 
(3) $\bm{z}_{\bm{\omega}}([1, \frac{3}{2}])$.}
\label{fig_trajectory}
\end{center}
\end{figure}

Recall that $o_{\bm{\omega}}$, $e_{\bm{\omega}}$ are braids for 
$\bm{\omega}= (\omega_1, \dots, \omega_{N-1}) \in \Omega_N$ 
as in (\ref{equation_3kinds_of_braids}). 
We claim that 
$\overline{b_{[ \epsilon, 1+ \epsilon]}}$ is written by 
$$\overline{b_{[ \epsilon, 1+ \epsilon]}}
= e_{-\bm{w}} \cdot o_{-\bm{w}} 
= \prod_{\substack{i \in \{1, \dots,N-1\} \\ i\ \text{even}  }} 
\sigma_i^{-\omega_i} 
\prod_{\substack{j \in \{1, \dots,N-1\} \\ j\ \text{odd}  }}
\sigma_j^{-\omega_j}.$$ 
See (\ref{equation_3kinds_of_braids}) for braids 
$e_{-\bm{w}}$ and $ o_{-\bm{w}} $. 
In fact, if $j$ is odd and $\omega_j= 1$ (resp. $\omega_j= -1$), 
then the braid word $\sigma_j^{-1}$ (resp. $\sigma_j^{+1}$) in $o_{-\bm{w}}$ 
appears in $\overline{b_{[ \epsilon, 1+ \epsilon]}}$
by the fact that 
$z_0(t)$ and $z_{N-j}(t)$ meet at the same vertical line at 
$t= \frac{j}{2}$ with the inequality 
$\mathrm{Im}(z_0(\tfrac{j}{2})) > \mathrm{Im}(z_{N-j}(\tfrac{j}{2}))$ 
(resp. $\mathrm{Im}(z_0(\tfrac{j}{2})) < \mathrm{Im}(z_{N-j}(\tfrac{j}{2}))$)
Similarly, 
if $i$ is even and $\omega_i= 1$ (resp. $\omega_i= -1$), 
then the braid word $\sigma_i^{-1}$ (resp. $\sigma_i^{+1}$) in  
$e_{-\bm{\omega}}$ 
appears in $\overline{b_{[ \epsilon, 1+ \epsilon]}}$. 
See Figures~\ref{fig_construction}(3).

By Lemma~\ref{lem_oe-alpha}, 
$ e_{-\bm{w}}o_{-\bm{w}} (= \overline{b_{[ \epsilon, 1+ \epsilon]}}) $ is conjugate to $\alpha_{- \bm{\omega}}$. 
This completes the proof of Claim. 
\medskip

By Claim, the primitive braid type of the simple choreography 
$\bm{z}_{\bm{\omega}}(t)$
is given by $\alpha_{-\bm{\omega}}$. 
Thus the braid type $Z_{\bm{\omega}}$ of $\bm{z}_{\bm{\omega}}(t)$ 
is given by $(\alpha_{- \bm{\omega}})^N$. 
This completes the proof of Theorem~\ref{thm_braid-type-omega}. 
\end{proof}

We are now ready to prove Theorem~~\ref{thm_main}. 

\begin{proof}[Proof of Theorem~\ref{thm_main}]
Given $\bm{\omega} \in \Omega_N$, 
we consider the simple choreography $\bm{z}_{- \bm{\omega}}(t)$ corresponding to 
$- \bm{\omega} \in \Omega_N$. 
By Theorem~\ref{thm_braid-type-omega}, 
$\alpha_{- (- \bm{\omega})} = \alpha_{\bm{\omega}}$ 
(resp. the $N$th power 
$(\alpha_{\bm{\omega}})^N$) represents the primitive braid type 
(resp. braid type $Z_{- \bm{\omega}}$) of the solution $\bm{z}_{- \bm{\omega}}(t)$. 
This together with Theorem~\ref{thm:yu_submain} implies 
the former statement of Theorem~\ref{thm_main}.

We turn to the latter statement. 
By Lemma~\ref{lem_stretchfactor-equivalent}, 
we may assume that 
$\bm{\omega}= (1, \omega_2, \dots, \omega_{N-1}) \in \Omega_N^+$. 
Suppose that  $\bm{\omega}=(1, 1, \dots, 1)$.  
In this case, 
by Example~\ref{ex_correspondence}(2) and Lemma~\ref{lem_periodic-braid}, 
the braid type $Z_{\bm{\omega}}$ is periodic.

Suppose that 
there exists $i \in \{2, \dots, N-1\}$ 
such that $\omega_i= -1$. 
Let $\Theta: \Psi_{N-1} \rightarrow \Omega_N^+$ 
be the bijection given in Lemma~\ref{lem_bijection}. 
Then $\Theta^{-1}(\bm{\omega}) \in \Psi_{N-1}$ is a composition of $N-1$ 
corresponding to the $(k+1)$-tuple of positive integers with $k>0$. 
Recall that $\beta_{\Theta^{-1}(\bm{\omega})} = \alpha_{\bm{\omega}}$ 
by Lemma~\ref{lem_alpha-beta}, 
and this is a pseudo-Anosov braid by Theorem~\ref{thm_recursive-example}. 

The $N$th power $(\alpha_{\bm{\omega}})^N$ that represents the braid type $Z_{\bm{\omega}}$ is also pseudo-Anosov, 
since the pseudo-Anosov property is preserved under the power. 
This completes the proof. 
\end{proof}

We recall the elements 
$\bm{\omega}_{\max}, \bm{\omega}_{\min} \in \Omega_N$ 
defined in Section~\ref{section_introduction}. 
The element $\bm{\omega}_{\min}$ can be written by 
$\bm{\omega}_{\min}=  \Theta(\bm{m})$ 
by using the bijection $\Theta: \Psi_{N-1} \rightarrow \Omega_N^+$, 
where 
$\bm{m}= (n,n) \in \Psi_{N-1}$ if $N= 2n+1$ and 
$\bm{m}= (n,n-1) \in \Psi_{N-1}$ if $N=2n$.

\begin{proof}[Proof of Theorem~\ref{thm_omega_max-min}]
Suppose that $\bm{\omega} \in \Omega_N$ is neither 
$(1,1, \dots, 1)$ nor $(-1,-1, \dots, -1)$. 
Then the braid $\alpha_{\bm{\omega}} \in B_N$ is pseudo-Anosov 
by Lemmas~\ref{lem_alpha-beta} and  \ref{lem_stretchfactor-equivalent}. 
By Lemma~\ref{lem_stretchfactor-equivalent} again, 
$\alpha_{\bm{\omega}}$ and $\alpha_{-\bm{\omega}}$ have 
the same stretch factor. 
By definitions of $\bm{\omega}_{\max}$ and $\bm{\omega}_{\min}$, 
we have 
$\alpha_{\bm{\omega}_{\max}}= \beta_{\bm{1}_{N-1}}$. 
Moreover 
$\alpha_{\bm{\omega}_{\min}} = \beta_{(n,n)}$ 
when $N= 2n+1$ 
(resp. $\alpha_{\bm{\omega}_{\min}}= \beta_{(n, n-1)}$ when $N= 2n$). 
The assertion follows from 
Theorems~\ref{thm_max-mini} and \ref{thm_braid-type-omega}. 
\end{proof}

\begin{example}
\label{ex_figure-8}
Suppose that $\bm{\omega} = (1, -1) \in \Omega_3$. 
The braid type of the figure-eight solution of the planar Newtonian $3$-body problem 
is the same as that of the simple choreography $\bm{z}_{\bm{\omega}}(t)$. 
See Remark~\ref{rem_figure-super-8}(1)(3). 
Hence, $(\alpha_{- \bm{\omega}})^3= (\sigma_1^{-1} \sigma_2)^3$ 
is a representative of the braid type of the figure-eight solution. 
By (\ref{equation_stretchfactor-power}), 
its stretch factor $\lambda((\sigma_1^{-1} \sigma_2)^3)$ 
equals $(\lambda(\sigma_1^{-1} \sigma_2))^3= (\tfrac{3+ \sqrt{5}}{2})^3$. 
\end{example}

\begin{corollary}
\label{cor_super8} 
The super-eight solution of the planar Newtonian $4$-body problem 
has the pseudo-Anosov braid type with the stretch factor 
$(2+ \sqrt{3})^4$. 
\end{corollary}

\begin{proof}
We take $\bm{\omega}= (1, -1, 1) \in \Omega_4$. 
Figure~\ref{fig_super-eight} illustrates the braid 
$\overline{b_{[\epsilon, 1+ \epsilon]}}$ 
corresponding to the simple choreography $\bm{z}_{\bm{\omega}}(t)$. 
The (primitive) braid type of the super-eight solution 
 for the planar $4$-body problem 
is the same as that of $\bm{z}_{\bm{\omega}}(t)$. 
See Remark~\ref{rem_figure-super-8}(2)(3). 
Thus the braid type of the super-eight is pseudo-Anosov by 
Theorem~\ref{thm_main}. 
The primitive braid type of the super-eight is given by 
$\alpha_{- \bm{\omega}}= \sigma_1^{-1} \sigma_2 \sigma_3^{-1}$. 
By Lemma~\ref{lem_stretchfactor-equivalent}, 
braids $\alpha_{- \bm{\omega}}$ and 
$\alpha_{\bm{\omega}}= \beta_{(1,1,1)}$ have the same stretch factor 
that is equal to $2+ \sqrt{3}$ 
(see Example~\ref{ex_111}(1)).  
Therefore, by (\ref{equation_stretchfactor-power}), 
the braid type of the super-eight  has the stretch factor $(2+ \sqrt{3})^4$. 
\end{proof}

\begin{figure}[htbp]
\begin{center}
\includegraphics[width=4.5in]{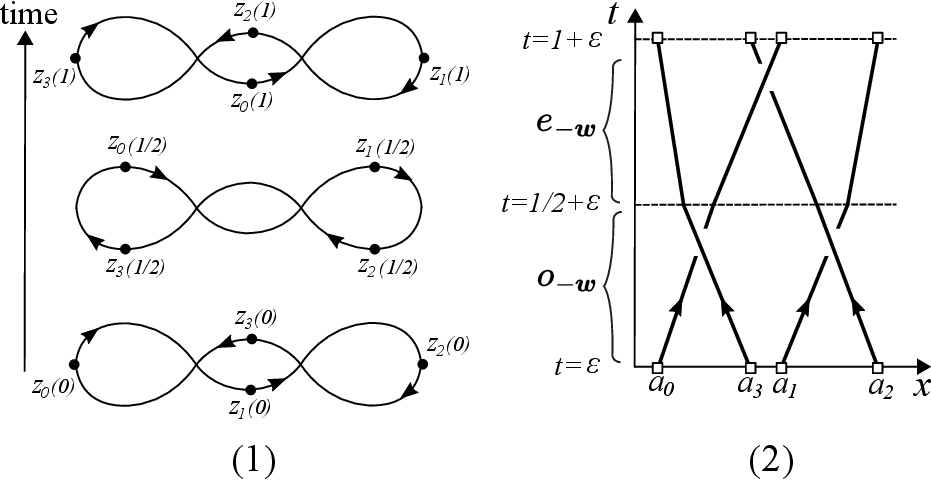}
\caption{
Case $N=4$, $\bm{\omega}= (1, -1,1)$. 
(1) From the bottom to the top: 
$\bm{z}_{\bm{\omega}}(0)$,  $\bm{z}_{\bm{\omega}}(\frac{1}{2})$ and 
$\bm{z}_{\bm{\omega}}(1)= \bm{z}_{\bm{\omega}}(0)$.
(2) 
$\overline{b_{[\epsilon, 1+ \epsilon]}}= e_{- \bm{\omega}} \cdot 
o_{- \bm{\omega}}= \sigma_2 \cdot \sigma_1^{-1} \sigma_3^{-1}$ in this case.}
\label{fig_super-eight}
\end{center}
\end{figure}

\begin{proof}[Proof of Corollary~\ref{cor_minimal-4braid}]
We consider the case $N=4$, $\bm{\omega} =(1, -1, -1)$. 
Figure~\ref{fig_minimal4-braid} illustrates the braid 
$\overline{b_{[\epsilon, 1+ \epsilon]}}$ 
corresponding to  $\bm{z}_{\bm{\omega}}(t)$. 
See also Figure~\ref{fig_orbit_minimal4-braid}. 
The primitive braid type of the simple choreography $\bm{z}_{\bm{\omega}}$ 
is given by 
$\alpha_{- \bm{\omega}} = \sigma_1^{-1} \sigma_2 \sigma_3$. 
If we take $- \bm{\omega} = (-1,1,1) \in \Omega_4$, 
the braid $\alpha_{- (- \bm{\omega})} = \alpha_{\bm{\omega}} = \sigma_1 \sigma_2^{-1} \sigma_3^{-1}$ represents the primitive braid type of 
the simple choreography $\bm{z}_{-\bm{\omega}}(t)$ corresponding to 
$-\bm{\omega} $. 
Since by Lemma~\ref{lem_stretchfactor-equivalent}, 
braids $\alpha_{-\bm{\omega}}$ and $\alpha_{\bm{\omega}} = \beta_{(1,2)}$ 
have the same stretch factor,  
Theorem~\ref{thm_recursive-example} tells us that 
$\lambda(\alpha_{-\bm{\omega}}) 
= \lambda(\beta_{(1,2)})\approx 2.2966$ 
is the largest real root of 
$$F_{(1,2)}(t)= (t+1)(t^4-2t^3-2t+1),$$ 
that is the largest real root of the degree $4$ polynomial 
 $t^4-2t^3-2t+1$. 
Thus, the simple choreography $\bm{z}_{-\bm{\omega}}(t)$  
has the desired properties. 
\end{proof}

\begin{figure}[htbp]
\begin{center}
\includegraphics[width=4.5in]{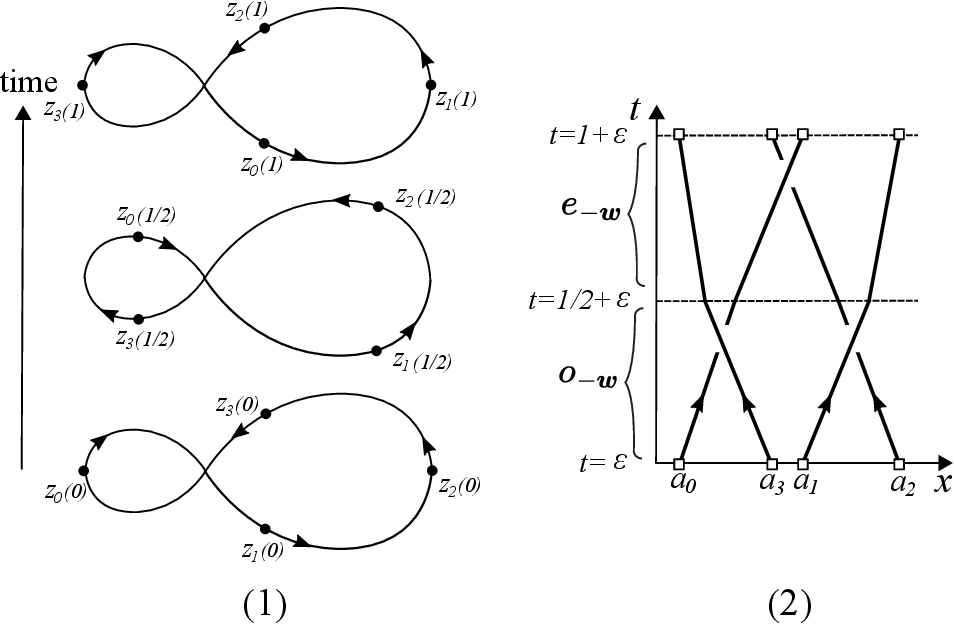}
\caption{
Case $N=4$, $\bm{\omega}= (1, -1,-1)$. 
(1) From the bottom to the top: 
$\bm{z}_{\bm{\omega}}(0)$,  $\bm{z}_{\bm{\omega}}(\frac{1}{2})$ and 
$\bm{z}_{\bm{\omega}}(1)= \bm{z}_{\bm{\omega}}(0)$.
(2) $\overline{b_{[\epsilon, 1+ \epsilon]}}= e_{- \bm{\omega}} \cdot 
o_{- \bm{\omega}}= \sigma_2 \cdot \sigma_1^{-1} \sigma_3$ in this case.
}
\label{fig_minimal4-braid}
\end{center}
\end{figure}

\section{Conclusion}
\label{section_conclusion}

We notice that 
if $b \in B_N$ is a primitive braid of some simple choreography 
of  the planar $N$-body problem, 
then the permutation $\hat{s}(b) \in S_N $ is cyclic
(see (\ref{equation_symmetrygroup}) for the definition of the homomorphism 
$\hat{s}: B_N \to S_N$). 
The following question 
is a choreographic version of Question~\ref{question_Montgomery}. 

\begin{question}
\label{question_primitive-braid}
Let  $b $ be a braid with $N$ strands 
whose permutation $\hat{s}(b)$ is cyclic. 
Is the braid type $\langle b \rangle$ given by a primitive braid of a 
simple choreography of the planar Newtonian $N$-body problem? 
\end{question}

Theorem~\ref{thm_main} tells us that 
Question~\ref{question_primitive-braid} is true if 
a braid $b$ is of the form $b= \alpha_{\bm{\omega}}$ 
for $\bm{\omega} \in \Omega_N$. 
We are  far from solving Questions~\ref{question_primitive-braid} 
and \ref{question_Montgomery}, 
yet we present an interesting example in this work.
In fact, 
the simple choreography of the planar  $4$-body problem 
which satisfies the statement of  Corollary~\ref{cor_minimal-4braid} 
is intriguing  in the sense that 
the braid $\sigma_1 \sigma_2^{-1} \sigma_3^{-1}$  reaches 
the minimal stretch factor among all pseudo-Anosov braids with $4$ strands 
\cite{SongKoLos02}. 

Table~\ref{table_max-mini-stretch-factor} in 
Section~\ref{subsection_strech-factors} 
shows that 
$\displaystyle\min_{\substack{ \beta \in Y_5 \setminus \{\beta_{(4)}\}}} \lambda(\beta) = \lambda(\beta_{(2,2)}) \approx 2.01536$. 
We note that 
there exists a pseudo-Anosov $5$-braid 
whose stretch factor is smaller than that of the $5$-braid $\beta_{(2,2)}$. 
Ham-Song \cite{HamSong16} proved that 
$\sigma_1 \sigma_2 \sigma_3 \sigma_4 \sigma_1 \sigma_2 $ 
is a pseudo-Anosov braid which realizes the minimal stretch factor 
$\delta_5 \approx 1.7220$ 
among all pseudo-Anosov $5$-braids, and $\delta_5$ is the largest real root of 
$t^4-t^3-t^2-t+1$. 

We finally ask the following question. 

\begin{question}
Does there exist a simple choreography of the planar Newtonian 
$5$-body problem 
whose primtive braid type is given by a pseudo-Anosov $5$-braid 
with the minimal stretch factor $\delta_5 \approx 1.7220$?  
\end{question}

\bibliographystyle{alpha}
\bibliography{Braids}

\end{document}